\documentclass[11pt]{amsart}
\usepackage[a4paper,centering,left=3.2cm,right=3.2cm]{geometry}
\usepackage{amsfonts,amscd,amssymb,amsmath,amsthm,mathrsfs,multirow,xcolor,lscape,longtable,pbox,lipsum,afterpage,graphicx,threeparttable,comment}
\usepackage[thinlines]{easytable}
\usepackage{enumerate}
\usepackage{todonotes}
\usepackage[colorlinks,linkcolor=blue,anchorcolor=blue,citecolor=blue,backref=page]{hyperref}
\usepackage{hyperref}

\allowdisplaybreaks
\newtheorem{proposition}{Proposition}[section]
\newtheorem{lemma}[proposition]{Lemma}
\newtheorem{corollary}[proposition]{Corollary}
\newtheorem{theorem}[proposition]{Theorem}

\newtheorem{definition}[proposition]{Definition}
\theoremstyle{definition}

\newtheorem{remark}[proposition]{Remark}

\numberwithin{equation}{section}

\usepackage[all,cmtip]{xy}
\usepackage{tikz}
\usetikzlibrary{arrows} 
\usetikzlibrary{matrix}
\usetikzlibrary{calc}

\DeclareMathOperator{\Cl}{Cl}
\DeclareMathOperator{\SL}{SL}
\DeclareMathOperator{\GL}{GL}
\DeclareMathOperator{\SO}{SO}
\DeclareMathOperator{\Sp}{Sp}
\DeclareMathOperator{\SU}{SU}

\DeclareMathOperator{\mdeg}{mdeg}
\DeclareMathOperator{\Ad}{Ad}
\DeclareMathOperator{\ad}{ad}

\begin{document}
\title[New dimensional estimates for subvarieties of linear algebraic groups]{New dimensional estimates for subvarieties of linear algebraic groups}
\date{\today}

\author{Jitendra Bajpai}
\address{Department of Mathematics, Kiel University, 24118 Kiel, Germany}
\email{jitendra@math.uni-kiel.de}

\author{Daniele Dona}
\address{Alfr\'ed R\'enyi Institute of Mathematics, 1053 Budapest, Hungary}
\email{dona@renyi.hu}

\author{Harald Andr\'es Helfgott}
\address{Mathematisches Institut, Georg-August-Universit\"at G\"ottingen, 37073 G\"ottingen, Germany; IMJ-PRG\\ Université Paris-Cité\\Boîte Courrier 7012\\8 place Aurélie Nemours\\75205 Paris Cedex 13\\France}
\email{harald.helfgott@gmail.com}

\begin{abstract}
  For every connected, almost simple linear algebraic group $G\leq\mathrm{GL}_{n}$ over a large enough field $K$, every subvariety $V\subseteq G$, and every finite generating set $A\subseteq G(K)$, we prove a general {\em dimensional bound}, that is, a bound of the form
  \[|A\cap V(\overline{K})|\leq C_{1}|A^{C_{2}}|^{\frac{\dim(V)}{\dim(G)}}\] with $C_{1},C_{2}$ depending only on $n,\deg(V)$. The dependence of $C_1$ on $n$ (or rather on $\dim (V)$) is doubly exponential, whereas $C_2$ (which is independent of $\deg(V)$) depends simply exponentially on $n$.
   
   Bounds of this form have proved useful in the study of growth in linear algebraic groups since 2005 (Helfgott) and, before then, in the study of subgroup structure (Larsen-Pink: $A$ a subgroup). In bounds for general $V$ and $G$ available before our work, the dependence of $C_1$ and $C_2$ on $n$ was of exponential-tower type.
 
   We draw immediate consequences regarding diameter bounds for untwisted classical groups $G(\mathbb{F}_{q})$. (In a separate paper, we derive stronger diameter bounds from stronger dimensional bounds we prove for specific families of varieties
   $V$.)
\end{abstract}

\maketitle
\tableofcontents  


\section{Introduction}\label{se:intro}

In the context of the study of growth of sets in algebraic groups, a {\em dimensional estimate} is an inequality that tells us that, if a set $A$ does not grow rapidly under multiplication by itself (that is, if it is what
some call an ``approximate subgroup''), then the intersection of $A$ with any
subvariety is not much larger than what one would guess very naïvely from the variety's dimension. 

More precisely: let $G\leq\mathrm{GL}_{n}$ be an algebraic group, seen as a variety inside the space of $n\times n$ matrices, let $V$ be a variety inside $G$, and consider a finite set $A\subseteq G(K)$ for some field $K$. A {\em dimensional estimate} is an inequality of the form
\begin{equation}\label{eq:dimest}
|A\cap V(\overline{K})|\leq C_{1}|A^{C_{2}}|^{\frac{\dim(V)}{\dim(G)}},
\end{equation}
where $C_{1},C_{2}$ depend only on the dimension and degree of $G$ and $V$, and not on $A$ or $K$. The notation $A^{k}$ indicates the set of all products of $k$ elements of $A$:
\begin{equation*}
A^{k}=\{a_{1}\cdot a_{2}\cdot\ldots\cdot a_{k}:a_{i}\in A\}.
\end{equation*}
We write $A^{-1}$ for $\{a^{-1}: a\in A\}$.

\subsection{Context: Origins of dimensional estimates}
\subsubsection{The study of subgroup structure}
In the special case of $A$ a subgroup $\Gamma$ of $G(K)$, dimensional estimates appear first in the work of Larsen and Pink~\cite{LP11}\footnote{First available in manuscript form in 1998.}.
 (Of course, $\Gamma$ cannot generate $G(K)$; instead, Larsen-Pink assume $\Gamma$ is 
 ``sufficiently general'', meaning essentially that $\Gamma$ is not contained in any algebraic subgroup of $G$ of low degree; see \cite[p. 1107]{LP11}.) Larsen and Pink called inequalities of the form
$|\Gamma \cap V(K)|\leq C |\Gamma|^{\dim(V)/\dim(G)}$ 
{\em nonconcentration estimates}.
The name {\em Larsen-Pink inequality} also appears in the literature.

 The main aim in~\cite{LP11} was to describe all finite subgroups $\Gamma\leq\mathrm{GL}_{n}(K)$ for $K$ an arbitrary field, thus generalizing Jordan's theorem to arbitrary characteristic -- without using the classification of finite simple groups. 
Dimensional estimates have found further applications: for instance, Larsen ~\cite{Lar04} used them to prove the dominance of word maps for simple algebraic groups. 

Since \cite{LP11} circulated in manuscript form for a long time before publication, the first published proof of the Larsen-Pink estimates was actually given by Hrushovski and Wagner in the setting of model theory~\cite{HW08}.

\subsubsection{Growth in linear algebraic groups}
A basic task in additive combinatorics is to classify all finite
subsets $A$ of an abelian group such that $A^{2}$ is not much larger than $A$. 
In the non-abelian setting, it makes more sense to classify subsets
$A$ such that $A^3$ is not much larger than $A$
(as that implies that $A^k$ is not much larger than $A$ for any bounded $k$). Helfgott proved in~\cite{Hel08} that any finite subset $A$ of $G=\SL_2(\mathbb{F}_p)$ generating $G$
satisfies a {\em product theorem}, i.e.\ a statement of the form
\begin{equation}\label{eq:prodthm}
\text{either $|A^3|\geq c |A|^{1+\eta}$ or $A^\ell = G$,}
\end{equation}
where $c,\eta,\ell$ are constants independent of $A$ and $p$. An immediate consequence of~\eqref{eq:prodthm} is a {\em diameter bound}, i.e.\ a statement of the form
\begin{equation}\label{eq:diam}
\textrm{diam}(G,A) \leq C(\log|G|)^{k},
\end{equation}
which is proved in~\cite{Hel08} with $C,k$ independent of $A$ and $p$. We recall that the
{\em diameter} $\textrm{diam}(G,A)$ of a finite group $G$ with respect to a subset $A\subseteq G$ generating $G$ can be defined to be the minimal $m$ such that
$(A\cup \{e\})^m = G$. This definition is equivalent to the usual definition by means of Cayley graphs; for a result like \eqref{eq:diam} it does not matter whether the graph diameter is directed or undirected, thanks to \cite[Thm.~1.4]{Bab06}.

The later paper~\cite{Hel11} (which appeared as a preprint in 2008), in the course of
giving a full proof of~\eqref{eq:prodthm} and~\eqref{eq:diam}
for $\SL_3$, established a dimensional estimate of the form
\eqref{eq:dimest} for any ``classical Chevalley'' group $G$ (i.e.\ untwisted classical) over
any field of characteristic $\ne 2$ -- in the
special case of $V$ a maximal torus \cite[Cor.~5.4]{Hel11}.
The procedure was actually stated there in somewhat greater generality (\cite[Prop.~4.12]{Hel11}) and applied to other varieties $V$ as well
(\cite[Cor.~5.14, Prop.~9.2, Prop.~9.4]{Hel11}). These were bounds
for $A\subseteq G(K)$ an arbitrary finite subset, as opposed to the special case of $A$ a subgroup as in Larsen-Pink.

There were two simultaneous, independent generalizations of \cite{Hel08} and \cite{Hel11}, namely, \cite{BGT11} and \cite{PS16} (summaries of which appeared as preprints in January 2010). Both give dimensional estimates for arbitrary subvarieties $V$ of any simple group $G$ of Lie type. The dimensional estimate in \cite{BGT11} is clearly of type \eqref{eq:dimest}; its proof follows the general lines of that in Larsen-Pink, adapted to sets rather than subgroups.
The proof in \cite{PS16} uses a different language and a somewhat different induction procedure, but leads to a
dimensional estimate of essentially the same shape.
Both papers used their dimensional estimates (that is, their versions of \eqref{eq:dimest}) to prove statements as in \eqref{eq:prodthm} and \eqref{eq:diam}.

\subsubsection{Approximate subgroups in model theory}
We have already mentioned \cite{HW08}. A little later,
Hrushovski proved \cite{Hru12} (available as a preprint in 2009) that, for any semisimple algebraic group $G$, if 
a finite $A\subseteq G(K)$ satisfies $|A^3|\leq C |A|$ for 
$C$ a constant, then a large proportion of the elements of
$A$ lie in a subgroup $H$ of $G(K)$ such that either $H$ is the
set of $K$-valued points of a proper algebraic subgroup of $G$ of bounded degree, or $H$ is contained in $(A^{-1} A)^2$.

The connection between Larsen-Pink and dimensional estimates as in~\cite{Hel08} and~\cite{Hel11} became apparent after this model-theoretic work. 
Pyber and Szab\'o credit the earlier work~\cite{HruPil} together with~\cite{Hel08} and~\cite{Hel11} as inspiration for their dimensional estimates.

\subsection{Results. Dependency on the rank.}

The results in~\cite{BGT11} and~\cite{PS16} depend on the rank of the linear algebraic group $G$. 
There has to be {\em some} dependency in dimensional bounds of type \eqref{eq:dimest} or \eqref{eq:prodthm}; a counterexample given in~\cite[Ch.~12]{PS16} shows that a result of the type~\eqref{eq:prodthm} 
cannot be in general true with $c$ and $\eta$ independent of the rank $r$ of $G$, and essentially the same counterexample shows the same for results of type~\eqref{eq:dimest}.

In~\cite{BGT11}, the dependence of the constants of \eqref{eq:dimest} on $r$ is completely non-explicit, due in part to the use of ultrafilters. In \cite{PS16}, the constants are given by towers of exponentials of height depending on $r$, as is clear from~\cite[Lem.~79]{PS16}. In both~\cite{BGT11} and~\cite{PS16}, the dependence on $r$ in the dimensional estimates~\eqref{eq:dimest} gets passed to growth results of type \eqref{eq:prodthm} and diameter bounds of type~\eqref{eq:diam}: the constants in these results are again either ineffective (in the case of \cite{BGT11}) or are of exponential-tower type (in the case of \cite{PS16}). 

We give general dimensional estimates whose constants have doubly and simply exponential dependence on $r$.

\begin{theorem}\label{th:main}
Let $G\leq\mathrm{GL}_{n}$ be a connected, almost simple linear algebraic group 
of rank $r$, defined over a field $K$ with $|K|\geq r+n\deg(G)$. Let $A\subseteq G(K)$ be a finite set with $\langle A\rangle=G(K)$ and $e\in A$.
Let $V$ be a subvariety of $G$ defined over $\overline{K}$. Write $d=\dim (V)$, $\delta = \dim (G)$. Then
\begin{equation}\label{eq:maintheq}
|A\cap V(\overline{K})|\leq C_{1} \left|A^{C_{2}}\right|^{d/\delta},
\end{equation}
where
\begin{align}\label{eq:manx}
C_{1} &\leq (2 \delta \deg(V))^{\delta^d}, & C_{2} & \leq n^{n^2+3}.
\end{align}

If $G$ is an untwisted classical group, the same conclusion holds without any assumption on $|K|$ and with $C_{2}\leq n^2 2^{n^2+1}$.
\end{theorem}

For a further estimate applying to $G=\mathrm{SL}_{n}$ (in its usual embedding inside $\mathrm{Mat}_{n}$), see Remark~\ref{re:slnusual}.
\smallskip

{\em Remarks on the statement.}
Here and elsewhere, $\langle A\rangle=G(K)$ means that $A$ generates $G(K)$ as a {\em semigroup}, i.e., every element of $G(K)$ can be written as a product of finitely many elements of $A$ (the identity element $e$ is the empty product). Similar results in the literature often require $A$ to only generate $G(K)$ as a group, but add the stronger condition $A=A^{-1}$; our hypothesis is weaker. Alternatively, we can apply Theorem~\ref{th:main} to $A\cup A^{-1}$ instead of $A$, and obtain \eqref{eq:maintheq} with $(A \cup A^{-1})^{C_2}$ instead of $A^{C_2}$ for any $A$ generating $G(K)$ as a group. If $K$ is finite then $A$ generates the same object as a group and as a semigroup, so there is nothing to argue. The condition $e\in A$ is used to simplify the treatment of the problem, as it implies that $A^{m}\supseteq A^{m'}$ whenever $m\geq m'$.

``Connected'' means ``connected in the Zariski topology'' (\S \ref{sss:verybasic}). This is a standard
condition in the study of algebraic groups. Without it, we could have $G$ be zero-dimensional and equal to any finite subgroup of $\GL_n(K)$. If $G$ is almost simple and of dimension $>0$, then it is connected.

``Almost simple'' here means what some authors call ``simple as an algebraic group''; see \S 
\ref{se:linalg}. In particular, abelian groups are not almost simple, and neither are, say, solvable
groups. It should be obvious that Theorem~\ref{th:main} is not true for $G<\GL_n$ the group of diagonal matrices, or of upper triangular matrices.

``Untwisted classical groups'' are properly defined in \S\ref{se:chev}. Essentially, they are $\SL_m$, $\SO_{2 m}^{+}$, $\SO_{2m+1}$ and $\Sp_{2 m}$ with mild restrictions on rank and characteristic.
The embedding of $\SL_m$ inside the affine spaces of square matrices $\mathrm{Mat}_n$ is different from the usual one, as it is given by the product of the usual representation and the contragredient representation (\S\ref{se:chev}).
We shall see that that embedding can be convenient, in that it leads to an improved bound on $C_2$, one that is also valid for $\SO$ and $\Sp$ (embedded in the standard way). See the discussion in \S\ref{se:chev} (on embeddings with small $\iota$) and Remark~\ref{re:slnusual}. The orthogonal groups $\SO_{2 m}^{+}$ and $\SO_{2m+1}$ are not in general almost simple, so the statement for untwisted classical groups is not just a particular case of the first part.
\smallskip

{\em Remarks on the method of proof.}
Our procedure is somewhat similar to Larsen-Pink's, in that we reduce the counting of elements of $A$ on $V$ to counting elements on fibres and images through appropriate morphisms. There are differences nonetheless, which we will discuss below. The dependence on the rank $r$ here seems to be roughly the natural limit of the method, qualitatively speaking, which is of course not to say that the bounds are tight. The counterexample in \cite[Ch.~12]{PS16} only shows that $C_{2}$ has to be greater than a constant times $r^2$.

As we can see, Theorem~\ref{th:main} is quantitatively much stronger than \cite{BGT11} and \cite{PS16}.
The reason lies mainly in the following: (a) better general tools for bounding the degree of varieties, (b) improvements in a procedure called {\em escape from subvarieties} in \S\ref{se:escape}, (c) a degree bound for the {\em exceptional locus} of a map in \S\ref{se:sing}, (d) working over the Lie algebra in a key
step, and (e) changes in the main inductive procedure, in that we replace the procedure going back to Larsen-Pink by our own.

(a) Proving bounds on $|A\cap V|$ requires keeping track of the ``complexity'' of $V$ in some quantitative
sense. Our chosen sense of ``complexity'' is simply the degree of $V$, defined as the sum of the degrees of
its irreducible components. We then need good bounds on the degrees of $f(V)$ and $f^{-1}(V)$, for
any morphism $f$. See \S \ref{sss:degree}.

(b) A proof of escape from a subvariety $V$ as in \cite{Hel08} or \cite{Hel11} (based on \cite{EMO05}) leads to an escape procedure with a number of steps that is doubly exponential on the dimension of $V$: in turn, this fact would lead to a doubly exponential dependence of $C_{2}$ on the rank $r$, even if it were enough for us to escape only once. We use instead a much-improved escape procedure from \cite{BDH21}, guaranteeing a simply exponential number of steps.

(c) As for the degree of an exceptional locus $E$, \cite{BGT11} and \cite{PS16} give different treatments, leading in both cases to the bad dependence of their versions of \eqref{eq:maintheq} on $r$. In \cite[App.~A]{BGT11}, an ultrafilter argument makes the bound on $\deg(E)$ ineffective. In \cite[Claim~80]{PS16}, standard elimination theory gives a doubly exponential partition of $E$, and \cite[Lem.~79]{PS16} works by induction to give an overall degree bound that is an exponential tower. One might improve the procedure so as to use elimination theory only once, giving a doubly exponential bound for $\deg(E)$, but then that is a further reason why the dependency of $C_2$ on $r$ would be doubly exponential (and it would stay so even with the improved escape procedure we will actually use). Our proof on the other hand skips elimination theory entirely.

(d) One can often get a better solution to a problem by linearizing it. Already \cite{Hel11} passed often to the Lie algebra $\mathfrak{g}$ of $G$. We will pass to the Lie algebra in order to prove our own version of Lemma \ref{le:skewness} (``generic skewness'', a type of result dating to Larsen-Pink). This is as in \cite[Lem.~4.6]{Hel19b}, except we will not assume that $\mathfrak{g}$ is simple; by working both over $G$ and over $\mathfrak{g}$, we will be able to make do just with the assumption that $G$ is almost simple. We will obtain a result with good control on the degree and no additional assumptions.

(e) We change the main
inductive process that lies at the heart of the proof (\S \ref{se:dimest}). Larsen and Pink
set up a complicated inductive procedure descending to two base cases ($V=G$ and $V$ zero-dimensional). We on the other hand organize our induction in stages, with each stage, or large step, involving a decrease in dimension (Proposition~\ref{pr:oberstair}), so that we have at most $\dim(V)$ large steps.
If we followed Larsen-Pink more closely, we would 
get substantially worse bounds (though, if we put into practice most of our other improvements,
we would still get $C_1 = \exp(\exp(n^{O(1)}))$ and $C_2 = \exp(n^{O(1)})$; see \cite[Appendix]{BDH21}).

\subsection{Applications}\label{subs:appl}

The proof of~\eqref{eq:prodthm} in both~\cite{BGT11} and~\cite{PS16} falls essentially into two halves: (a) dimensional estimates, and (b) a pivoting argument.
Here ``pivoting'' is a kind of inductive argument; the name comes from~\cite{Hel11}, where it was abstracted from proofs of the sum-product theorem in finite fields (in particular, work of Bourgain). The pivoting argument in~\cite{PS16} and~\cite{BGT11} is essentially the same. The first version of~\cite{BGT11} had a somewhat more involved argument, which was simplified by Helfgott.  

The natural question is then: what happens if we use Theorem~\ref{th:main} as our dimensional estimate while keeping the pivoting argument essentially as in~\cite{BGT11} and~\cite{PS16}?

 In~\cite{BDH21},
we prove a diameter bound of the form
\begin{equation}\label{eq:amago}
\textrm{diam}(G,A)\leq(\log|G|)^{O(r^4 \log r)}.\end{equation}
Proving~\eqref{eq:amago} necessitated a change in strategy:
instead of aiming at general dimensional estimates, we aimed at optimizing dimensional estimates for certain kinds of varieties (tori and
Zariski closures of conjugacy classes). In several respects, the procedure in~\cite{BDH21}
represents a return to~\cite{Hel11}, where dimensional estimates were also determined by explicit construction for specific varieties.
One of the aims of the present paper is to clarify what can be obtained from carefully optimized dimensional estimates of a general kind,
while otherwise proceeding as in~\cite{BGT11} and~\cite{PS16}.
In \S\ref{se:diambounds} we do exactly that: using Theorem~\ref{th:main} (or rather
Theorem \ref{th:ober}, which is a little more general), we prove Theorem~\ref{th:growth}, which is a growth result essentially of the form \eqref{eq:prodthm} with $c\geq\exp(-\exp(O(r^2\log r)))$; 
then we prove 
Theorem~\ref{th:diam}, which is 
a bound of type~\eqref{eq:diam} with $C = \exp(\exp(O(r^2 \log r))$ and $k = O(r^2\log r)$.

Since we are especially interested in groups over finite fields, it is natural to ask whether Theorem~\ref{th:main} can apply to other finite simple groups of Lie type. This is not immediate, since the twisted groups are not directly defined as the $\mathbb{F}_{q}$-points of an algebraic group. In \S\ref{se:concl} we offer some remarks in this direction.


\section{Basic tools and preliminaries}\label{se:prelim}

Let us collect basic properties about varieties, degrees, and $\mathbb{F}_{q}$-points, and recall a useful routine ({\em escape from subvarieties}) that will allow us to find generic elements when necessary.

\subsection{Varieties}
\subsubsection{Basic nomenclature}\label{sss:verybasic}

Our framework will be very classical -- more so than that in the first chapter of Hartshorne \cite{Hartshorne},
being closer to the first chapter of Mumford's Red Book \cite{Mumford}, say, or
\cite[Part II]{DS98}. Let
us go over some standard terms,
in part because their definitions in the literature can vary. 

A {\em variety} $V$ in $n$-dimensional affine space $\mathbb{A}^{n}$ is defined by a set of $s$ equations of the form $P_{i}(x_{1},\ldots,x_{n})=0$ (for $1\leq i\leq s$), where all $P_i$ are polynomials and $s$ is any non-negative integer. 
We say $V$ is {\em defined
over} a field $K$ if the coefficients of all $P_i$ 
belong to $K$. Define
the {\em set of points $V(K)$} to be
\begin{equation*}
V(K)=\{(k_{1},\ldots,k_{n})\in K^{n}:  P_{i}(k_{1},\ldots,k_{n})=0\;\;\forall 1\leq i\leq s\}.
\end{equation*}
Two sets of polynomials $\mathcal{P}=\{P_{i}\}_{i\leq s}$ and $\mathcal{Q}=\{Q_{j}\}_{j\leq t}$ with coefficients in $K$ define the same variety $V$ if the ideals generated by $\mathcal{P},\mathcal{Q}$ in the ring $\overline{K}[x_{1},\ldots,x_{n}]$ have the same radical, which we denote by $I(V)$. By the Nullstellensatz, $V=W$ if and only if $V(\overline{K})=W(\overline{K})$.
The intersection of two varieties defined by polynomials $\mathcal{P}=\{P_{i}\}_{i\leq s}$ and $\mathcal{Q}=\{Q_{j}\}_{j\leq t}$ is the variety defined by $\mathcal{P}\cup\mathcal{Q}$, while the union is defined by $\{P_{i}Q_{j}\}_{i\leq s, j\leq t}$. For $V$ defined over $K$, a \textit{subvariety} $W\subseteq V$ is a variety with $I(W)\supseteq I(V)$ (we do not necessarily require $W$ to be defined over $K$).

Note that what we have just defined is the so-called {\em set-theoretic} (or {\em reduced})  intersection of two varieties, not their scheme-theoretic intersection.
For instance, for us, the intersection of a parabola $y=x^2$ and the line $y=0$ is simply the point at the origin: the intersection is defined by
the polynomials $y-x^2$ and $y$, which generate an ideal whose radical is the
ideal generated by $x$ and $y$; hence, the intersection is considered to be
identical to the variety defined by $x=y=0$.

It will later become clear why the set-theoretic intersection is better suited
to our applications.

In the {\em Zariski topology}, the closed sets are the sets $V(\overline{K})$, $V$ a variety. Any topological term we use refers to this topology. The {\em Zariski closure} $\overline{S}$ of a set $S\subseteq \mathbb{A}^n(\overline{K})$ is its closure in the Zariski topology -- that is, $\overline{S}$ is the smallest set of the form $V(\overline{K})$ containing $S$. We may speak of
the variety $V$ itself as being the Zariski closure of $S$.

A variety $V$ is {\em irreducible}\footnote{Many sources, such as \cite{Hartshorne} and \cite{Mumford}, reserve the word {\em variety} for
what we call {\em irreducible variety}, and call {\em closed algebraic sets} what we call {\em varieties}. We follow instead the same convention that \cite{Shafv1} and \cite{DS98} follow.}
if it cannot be written as $V_{1}\cup V_{2}$ with $V_{1}\not\subseteq V_{2}$ and $V_{2}\not\subseteq V_{1}$. Every $V$ can be uniquely decomposed into a finite union of irreducible varieties not contained in each other, called the {\em irreducible components} of $V$. The {\em dimension} $\dim(V)$ of an irreducible variety $V$ is the largest $d$ for which we can write a chain of irreducible proper subvarieties $V_{0}\subsetneq V_{1}\subsetneq\ldots\subsetneq V_{d}=V$. We define the dimension of a non-irreducible variety $V$ to be the largest of the dimensions of its irreducible components. A variety is {\em pure-dimensional} when all its components are of the same dimension.

An {\em affine subspace} is a variety inside $\mathbb{A}^{n}$ that can be defined using only linear equations. A {\em linear subspace} is a variety defined using linear homogeneous equations.

We define projective space $\mathbb{P}^n$ as usual; we write $(x_0:x_1:\dotsc:x_n)$ for a point in $\mathbb{P}^n$.
We can define {\em projective varieties} in $\mathbb{P}^{n}$: a projective variety is given by a set of equations $P_i(x_0,x_1,\dotsc,x_n)=0$ for some {\em homogeneous} polynomials $P_i$.
All definitions above still apply -- although, in $\mathbb{P}^n$, we speak simply of {\em subspaces},
since the polynomials are always homogeneous. 
We can see $\mathbb{A}^{n}$ as an open subset of $\mathbb{P}^{n}$ in the Zariski topology -- for instance, by identifying points $(x_{1},\ldots,x_{n})$ of $\mathbb{A}^{n}$ with points of the form $(1:x_{1}:\ldots:x_{n})$ in $\mathbb{P}^{n}$. (More generally, one can embed $\mathbb{A}^n$ in $\mathbb{P}^n$ by
letting a linear combination $a_0 x_0 + \dotsc + a_n x_n$  equal $1$, say.)

There is a way to define a projective variety whose points are in one-to-one correspondence with subspaces of $\mathbb{P}^{n}$. For instance, there is a natural bijection between $(n-1)$-dimensional subspaces of $\mathbb{P}^{n}$ and points of $\mathbb{P}^{n}$ itself, in that $(y_{0}:y_{1}:\ldots:y_{n})\in\mathbb{P}^{n}$ defines the subspace $x_{0}y_{0}+\ldots+x_{n}y_{n}=0$. More generally, $d$-dimensional subspaces of $\mathbb{P}^{n}$ correspond to points of the {\em Grassmannian} $G(d+1,n+1)$, an irreducible projective variety.  This is the same as saying that $(d+1)$-dimensional linear subspaces of $\mathbb{A}^{n+1}$
correspond to points of $G(d+1,n+1)$.
Using the identification between $\mathbb{A}^{n}$ and an open subset of $\mathbb{P}^{n}$ described above, it is easy to show that $d$-dimensional {\em affine} subspaces in $\mathbb{A}^{n}$ correspond to points in an open subset of $G(d+1,n+1)$.

We say that a statement is true for a {\em generic point}\footnote{This is a classical meaning of ``generic'', used
in the literature on growth in groups (e.g., \cite[\S 2.5.2]{Hel11} and \cite[Def.~3.6]{BGT11}) but also elsewhere.
The meaning of ``generic'' in scheme theory is surely inspired by it but not identical:
the classical meaning is still current in conversation, but ``generic point'' in scheme theory means something else -- a non-closed point whose closure contains all closed points of a variety, i.e., all its points in the classical sense (see~\cite[\S II.1, Def.~1]{Mumford} and~\cite[\S II.2, Ex.~2.3.3]{Hartshorne}). Cf. the discussion at the beginning of~\cite[\S 2]{LP11}.}
of a variety $V$ if the set of points $x\in V(\overline{K})$ for which the statement is false is contained in a subvariety $W$ of $V$ with dimension lower than $\dim(V)$. If $V$ is irreducible, then $\dim(W)<\dim(V)$ if $W$ is a proper subvariety, i.e., if the complement of $W$ is nonempty. A statement is true for a {\em generic $d$-dimensional affine subspace} of $\mathbb{A}^{n}$ if it is true for a generic point of the Grassmannian $G(d+1,n+1)$.

A {\em morphism} $f:\mathbb{A}^{m}\rightarrow\mathbb{A}^{m'}$ is given by an $m'$-tuple of polynomials $f_{i}$ on $m$ variables.
A morphism $f:X\rightarrow Y$ for $X\subseteq\mathbb{A}^{m},Y\subseteq\mathbb{A}^{m'}$
is given as the restriction of a morphism 
$g:\mathbb{A}^{m}\rightarrow\mathbb{A}^{m'}$ such
that $g(x)$ lies on $Y$ for every point $x\in X(\overline{K})$.
If two distinct morphisms $g_{1},g_{2}$ as above agree on every point of
$X(\overline{K})$, then they are of course considered to define the same morphism $f$.\footnote{Ideal-theoretically: $g_1$ and $g_2$ are equivalent modulo $I(X)$. We are defining morphisms defined on $X$ as restrictions of morphisms $\mathbb{A}^m\to \mathbb{A}^{m'}$ so that there is no ambiguity
in speaking of the maximum degree $\mathrm{mdeg}(f)$ (\S\ref{sss:degree}).
}

For a morphism $f:X\to Y$ and a subvariety $V$ defined by polynomials
$P_i(y_1,\dotsc,y_n)$, the {\em preimage} $f^{-1}(Y)$ is defined by
the polynomials $P_i(f_1(\vec{x}),\dotsc,f_n(\vec{x}))$, together with
the polynomials defining $X$. The set of points
of $f^{-1}(Y)$ over a field $K$ consists of those points in $X(K)$
that are mapped by $f$ to points lying on $Y$. As always, we identify two
varieties if the corresponding ideals have the same radical. Thus, for instance,
if $X$ is the parabola $y-x^2=0$, $Y$ the $y$-axis $x=0$ and $f$ the projection
$\pi(x,y)=y$, the preimage $f^{-1}(\{0\})$ is, for us, simply the
variety given by $x=y=0$, i.e., the point at the origin, and not a double point at the origin. (In fact we do not work with
the notion of a ``double point'' at all.)

Let $X$ be a variety defined over a field $K$ and let $f:X\to Y$ be a morphism.
Then the image $f(X(\overline{K}))$ need not be the set of points of a variety,
though it is a {\em constructible set} (Chevalley; see
\cite[\S I.8, Cor.~2 to Thm.~3]{Mumford}), meaning a finite union of
intersections $U\cap W$, where $U$ is open and $W$ is closed. We can of course
study the Zariski closure $\overline{f(X(\overline{K}))}$, which is a variety
that we denote by $\overline{f(X)}$.

Morphisms are well-behaved in the sense that, for $f:X\rightarrow Y$ and $y\in Y$, the variety $f^{-1}(y)$ has generically the ``correct'' dimension. This fact is formalized in the theorems below, which are important results from algebraic geometry that include in particular the {\em upper semi-continuity of the dimension}. We recall that \cite[I.4, Def.~5]{Mumford} defines varieties to be irreducible, unlike here.

\begin{theorem}\label{th:chev}\cite[\S I.8, Thm.~2]{Mumford}
Let $f:X \rightarrow Y$ be a morphism of varieties over a field $K$, with $X$ irreducible. Let $u = \dim (X) - \dim(\overline{f(X)})$.

For every $y\in f(X(\overline{K}))$, every component of the variety $f^{-1}(y)$ has dimension $\geq u$. More generally, for any subvariety $W$ of $\overline{f(X)}$, the variety $f^{-1}(W)$ has a component of dimension $\geq \dim (W) + u$.
\end{theorem}

\begin{theorem}\label{th:chev2}
Let $f:X \rightarrow Y$ be a morphism of varieties over a field $K$, with $X$ irreducible. Let $u = \dim (X) - \dim(\overline{f(X)})$.
\begin{enumerate}[(i)]
\item\label{th:chev21} \cite[\S I.8, Cor.~1 to Thm.~3]{Mumford} There is a nonempty open set $U\subseteq \overline{f(X)}$ such that, for every $y\in U(\overline{K})$, the variety $f^{-1}(y)$ is nonempty and pure-dimensional of dimension $u$.
\item\label{th:chev23} \cite[\S I.8, Cor.~3 to Thm.~3 and Cor.~1 to Thm.~3]{Mumford} For any $x\in X(\overline{K})$, let $e(f,x)$ be the maximum among the dimensions of the irreducible components of $f^{-1}(f(x))$ containing $x$. Then
\begin{equation}\label{eq:exlocus}
\{x\in X(\overline{K}):e(f,x)>u\}
\end{equation}
is the set of points of a proper subvariety of $X$.
\end{enumerate}
\end{theorem}

\subsubsection{Degrees}\label{sss:degree}

Let $V\subseteq\mathbb{A}^{n}$ be a pure-dimensional variety with $\dim(V)=d$. By \cite[\S II.3.1.2, Thm.]{DS98}, for a generic $(n-d)$-dimensional affine subspace $L$ of $\mathbb{A}^n$ (in other words: 
for every $(n-d)$-dimensional affine subspace $L$ corresponding to a point in some
nonempty open subset $U\subseteq G(n-d+1,n+1)$), the intersection of
$V$ with $L$ consists of a finite, fixed number of points. We call that number the {\em degree} $\deg(V)$ of $V$.

We can extend the definition of degree to general varieties $V$: $\deg(V)$ is the sum of the degrees of the pure-dimensional parts of $V$.
The bound $\deg(V_{1}\cup V_{2})\leq\deg(V_{1})+\deg(V_{2})$ holds for any $V_{1},V_{2}$ directly by definition. 
If $V$ is the union of irreducible components $V_{i}$, then $\deg(V)=\sum_{i}\deg(V_{i})$.

By a generalization of \textit{B\'ezout's theorem}, for $V_{1},V_{2}$ pure-dimensional,
\begin{equation}\label{eq:bezout}
\deg(V_1\cap V_2)\leq \deg(V_{1})\deg(V_{2}).
\end{equation}
The classical form of Bézout's theorem is just a statement on curves on the plane.
The general version here is due to Fulton and Macpherson; it is based on their work
on intersection theory. The exact statement here can be found for instance in \cite[Ex.~8.4.6]{Ful84},
\cite[(2.26)]{zbMATH03880868}, or
\cite[\S II.3.2.2, Thm.]{DS98}; even more general and precise statements can be found in
Fulton~\cite{Ful84}. (For further proofs, see \cite{zbMATH03745352}, inspired by an earlier, and flawed, attempt by Severi, and \cite{zbMATH03966271}.)
By our definition of degree, we deduce easily from~\eqref{eq:bezout} that~\eqref{eq:bezout} itself holds for $V_1$, $V_2$ arbitrary, that is, not necessarily pure-dimensional.

{\em Another aside on intersections.} It is here that we see the effect of
working with set-theoretical intersections rather than scheme-theoretical intersections. If we were working with scheme-theoretical intersections, we would
need an additional condition (Cohen-Macaulay)
for the statement \eqref{eq:bezout} to be true.

It is possible to explain this issue entirely in terms of ideals rather than
schemes. Let us consider an example\footnote{Kindly shown to us by Miles Reid.} in $\mathbb{A}^4$.

Let $V_1$ be the union of the $xy$-plane through the origin and the $zw$-plane through the origin. Let $V_2$ be the hyperplane given by $x=z$. Then the set-theoretic intersection $V_1\cap V_2$ is the union of the line $x=z=w=0$ and the line
$x=y=z=0$. We see then that \eqref{eq:bezout} holds: $\deg(V_1\cap V_2)=2$, $\deg(V_1)=2$, $\deg(V_2)=1$.

To obtain the scheme-theoretic intersection, we should start by considering
the ideals $I_1$ and $I_2$ defining $V_1$ and $V_2$ (namely,
$I_1 = (x z, x w, y z, y w)$ and $I_2 = (x-z)$. We then take the sum
$I_1+I_2$, which is the ideal $(x z, x w, y z, y w, x-z)$. Now we will {\em not}
take the radical of this ideal. Rather, we take the primary decomposition
$I_1 + I_2 = Q_1\cap Q_2\cap Q_3$, where
$$Q_1 = (x,y,z),\;\;\;\;\;\; Q_2 = (x,z,w),\;\;\;\;\;\;
Q_3 = (x-z,y,z^2,w).$$
Here $Q_1$ and $Q_2$ just give us the lines $x=z=w=0$ and the line
$x=y=z=0$, whereas $Q_3$ is a somewhat fat version of the point at the
origin (since the radical of $Q_3$ is just $(x,y,z,w)$, which defines
the point at the origin). The scheme-theoretical intersection $V_1\cap V_2$
is defined to consist of those three varieties; its degree must then be
at least $3$. We see, then, that, for the scheme-theoretical intersection, the
generalized Bézout inequality \eqref{eq:bezout} can fail to hold.
{\em End of aside.}

Let $V$ be a variety defined by a single polynomial equation $P=0$ with $\deg(P)>0$. Then $\deg(V)\leq\deg(P)$, and equality holds if $P$ has no repeated factors. By B\'ezout, if $V$ is defined by many equations $P_{i}=0$, which means that $V=\bigcap_{i}V_{i}$ with $V_{i}$ defined by the single $P_{i}=0$, then $\deg(V)\leq\prod_{i}\deg(P_{i})$.

For a morphism $f:\mathbb{A}^{m}\rightarrow\mathbb{A}^{m'}$ given by an $m'$-tuple of polynomials $f_{i}$, we define the {\em maximum degree} of $f$ to be $\mdeg(f):=\max_{i}\deg(f_{i})$. This definition does not 
seem to be in the standard literature, but it is natural to work with this quantity when dealing with explicit bounds in algebraic geometry. 
For a morphism $f:X\rightarrow Y$, we define $\mdeg(f)$ to be the minimum of $\mdeg(g)$ over all $g:\mathbb{A}^{m}\rightarrow\mathbb{A}^{m'}$ with $g|_{X}=f$.

\subsubsection{Degrees of images and preimages}

Let us prove several basic bounds on how the degree of a variety behaves
under morphisms. We recall that we are working with what may be called
set-theoretic preimages, set-theoretic intersections, etc.

\begin{lemma}\label{le:zarimdeg}
Let $V\subseteq\mathbb{A}^m$ be a variety and $f:V\rightarrow\mathbb{A}^{n}$ a morphism. Then\footnote{Bound \eqref{eq:boundima} was given by user Angelo on MathOverflow in 2011 in reply to question 63451 by HH.}
\begin{equation}\label{eq:boundima}
\deg(\overline{f(V)})\leq\deg(V)\mdeg(f)^{\dim(\overline{f(V)})}.
\end{equation}
\end{lemma}
The bound is tight for Veronese varieties.
\begin{proof} 
We sketch the proof, which is given in~\cite{BDH21}. If $\dim(\overline{f(V)})<\dim(V)$,
then $\overline{f(H\cap V)}=\overline{f(V)}$ for a generic hyperplane $H$, so up to taking intersection with finitely many $H$ we may assume that $\dim(\overline{f(V)})=\dim(V)=d$. Now a generic fibre is made of finitely many points, and taking a generic affine subspace $L$ of codimension $d$ we get
\begin{equation*}
\deg(\overline{f(V)})=|L\cap\overline{f(V)}|\leq|f^{-1}(L)|\leq\deg(f^{-1}(L))\leq\deg(V)\mdeg(f)^{d}
\end{equation*}
by B\'ezout.
\end{proof}

In particular, since any projection $\pi$ has $\mdeg(\pi)=1$, we have $\deg(\overline{\pi(V)})\leq \deg(V)$.

\begin{lemma}\label{le:invimdeg-new}
Let $V\subseteq\mathbb{A}^{m}$ and $W\subseteq\mathbb{A}^{n}$ be varieties and $f:V\rightarrow\mathbb{A}^{n}$ a morphism. Then
\begin{equation}
\deg(f^{-1}(W))\leq\deg(V)\deg(W)\mdeg(f)^{\dim(\overline{f(V)})}.
\end{equation}
\end{lemma}

\begin{proof}
First we prove the weaker bound
\begin{equation}\label{eq:invimdeg}
\deg(f^{-1}(W))\leq\deg(V)\deg(W)\mdeg(f)^{\dim(V)}.
\end{equation}
Let $g:V\to V\times \mathbb{A}^n$ be given by $g(x) = (x,f(x))$ and let $\pi_1$ be the projection  $V\times \mathbb{A}^n \to V$. Then $f^{-1}(W) = \pi_1(g(V)\cap (\mathbb{A}^m\times W))$, and so
$$\begin{aligned}\deg(f^{-1}(W)) &\leq \deg(g(V) \cap (\mathbb{A}^m\times W))\\ &\leq \deg(g(V)) \deg(W) \leq \deg(V) \deg(W) \mdeg(f)^{\dim(V)} \end{aligned}$$
by B\'ezout and Lemma \ref{le:zarimdeg}. 

Now let $L\subseteq\mathbb{A}^{m}$ be an affine subspace $L$ of codimension $\mathrm{codim}(L) = \dim(f^{-1}(W))$, generic in the sense that $L\cap f^{-1}(W)$ is a set of $\deg(f^{-1}(W))$ points. Replace $V$ and $f$ with $L\cap V$ and $f|_{L\cap V}$, noting that $\deg(L\cap V)\leq\deg(V)$ by B\'ezout and $\mdeg(f|_{L\cap V})\leq\mdeg(f)$.  In other words, we have reduced matters to the case where 
$f^{-1}(W)$ is zero-dimensional. We may now also reduce to the case of $V$ irreducible without loss of generality.

Theorem~\ref{th:chev} and the irreducibility of $V$ imply that $\dim (f^{-1}(W))\geq \dim(V)-\dim(\overline{f(V)})$. Since $\dim (f^{-1}(W))=0$, we conclude that $\dim(V) = \dim(\overline{f(V)})$. The result now follows from \eqref{eq:invimdeg}.
\end{proof}
One can actually modify the proof of Lemma \ref{le:invimdeg-new} slightly to prove a stronger bound:
\begin{equation}\label{eq:gatos}
\deg(f^{-1}(W))\leq\deg(V)\deg(W)\mdeg(f)^{\dim(\overline{f(V)}) - \dim(\overline{f(V)\cap W})}.
\end{equation}
The argument is as follows. First, by choosing a generic $L$, we can assume not just that
$|L\cap f^{-1}(W)| = \deg(f^{-1}(W))$ but also that $L\cap V$ has dimension 
$\dim (V) - \mathrm{codim} (L) = \dim (V) - \dim(f^{-1}(W))$. As a result, we obtain the bound
$$\deg(f^{-1}(W))\leq\deg(V)\deg(W)\mdeg(f)^{\dim (V) - \dim(f^{-1}(W))}.$$
Again by Theorem~\ref{th:chev}, $\dim (f^{-1}(W)) \geq \dim(\overline{f(V)\cap W}) + \dim(V) - \dim(\overline{f(V)})$, and so 
$$\dim (V) - \dim(f^{-1}(W)) \leq \dim(\overline{f(V)})-
\dim(\overline{f(V)\cap W}).$$
Thus, \eqref{eq:gatos} holds.

It is easy to find cases in which \eqref{eq:gatos} is tight: 
most maps $f:V=\mathbb{A}^n\to \mathbb{A}^n$ given by polynomials of degree $D$ and most varieties
$W\subseteq \mathbb{A}^n$ that are complete intersections should give us examples (since, for $g_i:\mathbb{A}^n\to \mathbb{A}^1$ of degree $d_i$, and $W$ of degree $\prod_{i=1}^k d_i$ given by $g_i(\vec{y})=0$, $1\leq i\leq k$, the variety $f^{-1}(W)$ given by $(g_i\circ f)(\vec{y}) = 0$ should
have degree $\prod_{1\leq i\leq k} \deg(g_i\circ f) = D^k \prod_{i=1}^k d_i = 
D^{n-\dim(W)} \deg(W)$).

One can improve $\dim(\overline{f(V)\cap W})$ in \eqref{eq:gatos} to
$\dim(\overline{f(V)}\cap W)$, as follows. First, because \eqref{eq:invimdeg} holds for affine varieties, it also holds for projective varieties. (Just choose a copy of affine space inside $\mathbb{P}^m$ so that the complement does not contain any components of $f^{-1}(W)$ or $V$.)  Then proceed as above to prove $\deg(f^{-1}(W))\leq \deg(V) \deg(W) \mdeg(f)^{\dim(f(V))-\dim (f(V)\cap W)}$ for $V$, $W$ projective. This statement then implies that \eqref{eq:gatos} holds for $v$, $W$ affine with
$\dim(\overline{f(V)}\cap W)$ instead of $\dim(\overline{f(V)\cap W})$.

We will need that stronger version of \eqref{eq:gatos} only in the special case of $W$ a hypersurface, i.e., $W$ defined by a single polynomial equation $g(\vec{y})=0$ 
(where $g$ is not necessarily irreducible). In that case, the proof is much shorter, and in fact obvious.

\begin{lemma}\label{le:invimdeg-silly}
Let $V\subseteq\mathbb{A}^{m}$ and $W\subseteq\mathbb{A}^{n}$ be varieties,
with $W$ given by the equation $g(\vec{y}) = 0$, $g$ a polynomial. Let
$f:V\rightarrow\mathbb{A}^{n}$ a morphism. Then
\begin{equation}
\deg(f^{-1}(W))\leq\deg(V) \deg(W) \mdeg(f).
\end{equation}
\end{lemma}
\begin{proof}
 The variety $f^{-1}(W)$ is the intersection of $V$ with the hypersurface
 in $\mathbb{A}^m$ given by $(g\circ f)(x) = 0$. 
By Bézout, $\deg(g\circ f)\leq \deg(g) \mdeg(f)$ and $\deg(g) = \deg(W)$,
\begin{equation*}
\deg(f^{-1}(W))\leq \deg(V) \deg(g\circ f)\leq \deg(V) \deg(W) \mdeg(f). \qedhere
\end{equation*}
\end{proof}

We will use the following simple lemma in conjunction with Lemma \ref{le:invimdeg-silly}.
\begin{lemma}\label{le:inhypers}
Let $W,V\subseteq \mathbb{A}^n$ be affine varieties such that $W$ does not contain $V$. Then there is a hypersurface $W'\subseteq\mathbb{A}^n$ with $\deg(W')\leq \deg(W)$ containing $W$ but not $V$.
\end{lemma}
\begin{proof}
We can clearly assume that $V$ is irreducible. Then we can also
assume without loss of generality that $W$ is irreducible (as we can take
the union of hypersurfaces $W'$ at the end). Choose a point $x$ on $V$ lying outside $W$.

Say $\dim (W) = d < n-1$. The closure of the projection of $W$ in a generic direction has dimension $d$, and does not contain the projection of $x$. We iterate, and obtain that there is a projection $\pi:\mathbb{A}^n\to \mathbb{A}^{d + 1}$ with 
$\dim (\overline{\pi(W)}) = d$ and $\pi(x)$ lying outside $\overline{\pi(W)}$.
Since $\overline{\pi(W)}$ is of codimension $1$ in $\mathbb{A}^{d+1}$, it is
a hypersurface $g(y)=0$.
Then $\deg (g) = \deg (\overline{\pi(W)})$, and so,
 by Lemma \ref{le:zarimdeg}, $\deg (g) \leq \deg (W)$. Define $W'\subseteq \mathbb{A}^n$ by $(g\circ \pi)(y)=0$.
\end{proof}

\subsection{Linear algebraic groups}\label{se:linalg}

The space of $n\times n$ matrices $\mathrm{Mat}_{n}$ is the affine space $\mathbb{A}^{n^{2}}$ endowed with a {\em multiplication map}, i.e.\ a morphism $\cdot:\mathbb{A}^{n^{2}}\times\mathbb{A}^{n^{2}}\rightarrow\mathbb{A}^{n^{2}}$ defined by the usual matrix multiplication. Clearly, $\mathrm{mdeg}(\cdot)=2$. The {\em determinant} is a morphism $\det:\mathrm{Mat}_{n}\rightarrow\mathbb{A}^{1}$; it has maximum degree $\mathrm{mdeg}(\det)=n$.

The open set of $\mathrm{Mat}_{n}$ defined by $\det(x)\neq 0$ is the {\em general linear group} $\mathrm{GL}_{n}$. If necessary, we can work in a larger space so as to make $\mathrm{GL}_{n}$ a variety (i.e., Zariski-closed) and matrix inversion a morphism of $\mathrm{GL}_{n}$ into itself. For instance, we can identify $\mathrm{GL}_{n}$ with the variety of elements $(x,y)\in\mathrm{Mat}_{n}\times\mathbb{A}^{1}$ such that $\det(x)y=1$: the {\em inversion map} is then the morphism $^{-1}:\mathrm{GL}_{n}\rightarrow\mathrm{GL}_{n}$ sending $(x,y)$ to $(\mathrm{adj}(x)y,\det(x))$, where the {\em adjugate} $\mathrm{adj}(x)$ is the matrix whose $(i,j)$-th entry is $(-1)^{i+j}M_{ji}$, with $M_{ji}$ being the $(j,i)$-th minor of $x$. Note that multiplication and the determinant are still morphisms in this larger space, with the same maximum degree as before.

Following Borel~\cite[p.~51]{Bor91}, we define
a {\em linear algebraic group} $G$ to be any subvariety of $\mathrm{GL}_{n}$ closed under the multiplication and inversion maps. The maximum degree of these maps may change when restricted to $G$. The degree of $G$ and its maps may also differ depending on how we choose to represent $G$: for instance $\mathrm{SL}_{n}$ will be defined below so as to have $\mathrm{mdeg}(^{-1})=1$, whereas working with $\mathrm{SL}_{n}$ as a subvariety of $\mathrm{Mat}_{n}$ yields $\mathrm{mdeg}(^{-1})=n-1$. Throughout the rest of the paper, when dealing with a linear algebraic group we use
\begin{align}\label{eq:notation}
\delta & =\dim(G), & \Delta & =\deg(G), & \iota & =\mathrm{mdeg}(^{-1}).
\end{align}

A linear algebraic group $G$ is {\em almost simple}\footnote{Some authors call these groups {\em simple}.
``Such an algebraic group is called
{\em simple} (or {\em almost simple}, 
if we wish to emphasize that the group need not be simple as an abstract group)''
 (\cite[p.~168]{Hum95b}).
}
if it is non-abelian and has no connected, normal linear algebraic
subgroups except $\{e\}$ and $G$. (``Connected'' here means ``connected in the Zariski topology'', i.e.,
``not the union of two varieties with empty intersection''.) A moment's thought shows that having no non-trivial connected, normal linear algebraic subgroups is equivalent to having no normal linear algebraic subgroups of positive dimension. An almost simple linear algebraic group is by definition also {\em reductive} \cite[\S 11.21]{Bor91}, because the radical of $G$ is a connected normal linear algebraic subgroup of $G$. For a reductive group $G$, the {\em rank} $r$ is the dimension of a maximal torus (a {\em torus} inside a group $G$ defined over $K$ is a linear algebraic subgroup of $G$ isomorphic to $\mathrm{GL}_{1}^{m}$ over $\overline{K}$ for some $m$ \cite[\S 8.5]{Bor91}).

If $G\leq\mathrm{GL}_n$ is non-abelian, then $n\geq 2$, since $\mathrm{GL}_{1}$ is abelian. If a linear algebraic group is connected, then it is also irreducible~\cite[Cor.~1.35]{Mil17}. 
(This is the non-trivial direction; in any topology, an irreducible closed set is of course connected.)
If $G\leq\mathrm{GL}_{n}$ is connected and almost simple then $G\leq\mathrm{SL}_{n}$, by the following simple argument: $H=G\cap\mathrm{SL}_{n}$ is the kernel of the determinant map $\det|_{G}:G\rightarrow\mathbb{A}^{1}$, and thus, if $H\neq G$, then $\dim(H)=0$, because $H$ is normal in $G$ and $G$ is almost simple; since $H=(\det|_{G})^{-1}(1)$, we conclude by Theorem~\ref{th:chev} that $\dim(G)\leq 1$. 
Hence, $G$ must be either $\{e\}$, the additive group, or the multiplicative group \cite[\S 20]{Hum95a}, contradicting the fact that $G$ is non-abelian.

\subsection{Tangent spaces and adjoint representation}

Let us go over the basics.

Let $V\subseteq\mathbb{A}^n$ be an affine variety defined over $K$, and let $I(V)$ be as in \S\ref{sss:verybasic}. 
Let $x$ be a point on $V(\overline{K})$. It is clear that,
for every $y \in \mathbb{A}^n(\overline{K})$ and every $P$ in $I(V)$, the polynomial
$P_{y}(t) = P((1-t) x + t y)$ vanishes at $t=0$. The condition that the linear
term of $P_{y}(t)$ be $0$ for every $P\in I(V)$ is equivalent to $y$ lying on
an affine linear variety in $\mathbb{A}^n$, called the {\em tangent space} $T_x V$ to $V$ at $x$.  (The alternative would be to work with the Zariski tangent space, which is defined intrinsically; $T_x V$ (in our definition) and the Zariski tangent space are isomorphic \cite[Cor.~2.1]{Shafv1}.)

A point $x$ on $V$ is called {\em non-singular} if, for every irreducible component $W$ of $V$ containing $x$, $\dim (T_x W) = \dim (W)$.
If $V$ is irreducible, there is a nonempty open subset $U\subseteq V$ such that
$x$ is non-singular if and only if $x\in U(\overline{K})$ \cite[\S III.4, Prop.~2-3]{Mumford}. One can quickly
see that this statement is then true for $V$ reducible as well.
We say that $V$ is {\em non-singular} if every point $x$ on $V$ is non-singular. Algebraic groups are non-singular (as there is always a map $g\cdot$ taking the identity $e$ to a given point $g$). 

Given a morphism $f:V\to W$ for $V\subseteq \mathbb{A}^m$, $W\subseteq \mathbb{A}^n$,
formally differentiating the polynomials defining $f$ gives us a linear map 
$(D f)_x:T_x V\to T_{f(x)} W$, called the {\em derivative} or {\em differential} of
$f$ at $x$.

Let $G$ be a linear algebraic group over $K$. Write $\mathfrak{g}$ for
its tangent space at the origin $e$. Given $g\in G(\overline{K})$,
let $\Psi_g:G\to G$ be the map $x\mapsto g x g^{-1}$. Define $\Ad_g:\mathfrak{g}\to \mathfrak{g}$ to be the derivative of $\Psi_g$ at $x=e$.
We call $g\mapsto \Ad_g$ the {\em adjoint representation} of $G$. 

Write $\ad$ for the derivative of $\Ad_g$ at $g=e$; then $\ad(x)$ is a linear endomorphism of $\mathfrak{g}$. For $x,y\in \mathfrak{g}(\overline{K})$, let
$[x,y] = \ad(x)(y)$. Then $[\cdot,\cdot]$ makes 
$\mathfrak{g}(\overline{K})$ (and $\mathfrak{g}(K)$) into a Lie algebra.

\subsection{Untwisted classical groups}\label{se:chev}

An \textit{untwisted classical group} $G$ is for us one of the following: $\mathrm{SL}_{n},\mathrm{SO}_{2n}^{+},\mathrm{SO}_{2n+1},\mathrm{Sp}_{2n}$. The same groups are also referenced as ``classical Chevalley'' in \cite{Hel11}. Since the degrees of varieties and maps depend on how $G$ is embedded in the matrix space, we choose a convenient explicit embedding.

\begin{definition}\label{de:untwisted-classical}
An \textit{untwisted classical group} of rank $r$ over a field $K$ of characteristic $p\geq 0$ is one of the following affine varieties defined over $K$ under the corresponding restrictions:
\begin{align*}
\mathrm{SL}_{r+1} & =\left\{\begin{pmatrix} x_{1} & 0 \\ 0 & x_{2} \end{pmatrix}\in\mathrm{Mat}_{2r+2}:\det(x_{1})=1,x_{1}x_{2}^{\top}=\mathrm{Id}_{r+1}\right\} & & \text{($r\geq 1$),} \\
\mathrm{SO}_{2r}^{+} & =\{x\in\mathrm{Mat}_{2r}:\det(x)=1,x^{\top}M_{1}x=M_{1}\} & & \text{($p\neq 2$, $r\geq 4$),} \\
\mathrm{SO}_{2r+1} & =\{x\in\mathrm{Mat}_{2r+1}:\det(x)=1,x^{\top}M_{2}x=M_{2}\} & & \text{($p\neq 2$, $r\geq 3$),} \\
\mathrm{Sp}_{2r} & =\{x\in\mathrm{Mat}_{2r}:x^{\top}M_{3}x=M_{3}\} & & \text{($r\geq 2$),}
\end{align*}
where
\begin{align*}
M_{1} & =\begin{pmatrix} 0 & \mathrm{Id}_{r} \\ \mathrm{Id}_{r} & 0 \end{pmatrix}, &
M_{2} & =\begin{pmatrix} 0 & \mathrm{Id}_{r} & 0 \\ \mathrm{Id}_{r} & 0 & 0 \\ 0 & 0 & 1 \end{pmatrix}, &
M_{3} & =\begin{pmatrix} 0 & \mathrm{Id}_{r} \\ -\mathrm{Id}_{r} & 0 \end{pmatrix}.
\end{align*}
All of them are linear algebraic groups.
\end{definition}

\begin{remark}\label{re:untwisted}
Here are a few remarks on our choices in the definition above.
\begin{enumerate}[(i)]
\item All of these affine varieties are defined in such a way that the maximum degree $\iota$ of the inverse map is $1$. Clearly, one could define $\SL_{r+1}$ instead in the usual way, namely as the subvariety of $\mathrm{Mat}_{r+1}$ given by $\det(g)=1$, but $\iota$ would then be $r$. We will generally work with $\mathrm{SL}_{r+1}$ as in Definition~\ref{de:untwisted-classical}, but the usual embedding of $\SL_{r+1}$ is covered by the first part of Theorem~\ref{th:main} (since it is connected and almost simple). Remark~\ref{re:slnusual} describes further bounds for the usual embedding of $\SL_{r+1}$, with larger $C_{1}$ and smaller $C_{2}$ than in Theorem~\ref{th:main}.
\item\label{re:untwisted-minus} The matrices $M_{1},M_{2},M_{3}$ above are a standard choice, as in \cite[Prop.~2.5.3(i)]{KL90}, \cite[Prop.~2.5.3(iii)]{KL90}, and \cite[Prop.~2.4.1]{KL90} respectively. The plus sign in the notation $\mathrm{SO}_{2r}^{+}$ is necessary over the finite field $\mathbb{F}_{q}$, as a different choice of $M_{1}$ may lead to a different group $G(\mathbb{F}_{q})$; this does not happen for instance over algebraically closed fields, where all non-degenerate quadratic forms are equivalent. For comments about $\mathrm{SO}_{2r}^{-}$, which is twisted  (but may also be defined from a different $M_{1}$ as in \cite[Prop.~2.5.3(ii)]{KL90}), see \S\ref{se:concl}.
\item\label{re:untwisted-o2} We restrict to $p\neq 2$ for the orthogonal cases $\mathrm{SO}_{2r}^{+},\mathrm{SO}_{2r+1}$ for two distinct reasons relating to our interest in groups over $\mathbb{F}_{q}$. For the latter group, when $q$ is even we have $\mathrm{SO}_{2r+1}(\mathbb{F}_{q})\simeq\mathrm{Sp}_{2r}(\mathbb{F}_{q})$: see \cite[\S 3.4.7]{Wil09}. As for the former, $\mathrm{SO}_{2r}^{+}$ is not connected for $q$ even: in fact, it has two connected components, distinguished by the Dickson invariant (see \cite[\S II.10]{Die55} for details). All untwisted classical groups in our definition are connected \cite[Ex.~2.42-2.43-2.44]{Mil17}.
\item The reason behind the restrictions on the rank $r$ is that some of the groups taken out by these restrictions are not almost simple ($\SO_{2}^{+}$ is abelian, $\SO_{4}^{+}$ has the group of left- (or right-)isoclinic rotations as a normal subgroup) and the others are isomorphic to other groups already considered ($\Sp_2$ is isomorphic to $\SL_2$, and $\SO_3(\mathbb{F}_q)$ has a subgroup $\Omega_3(\mathbb{F}_q)$ of index $2$ isomorphic to $\mathrm{PSL}_2(\mathbb{F}_q)$). See, e.g., \cite[\S 1.2]{Wil09} or \cite[Ch.~2]{zbMATH03906699}.
\item\label{re:untwisted-as} $G$ is almost simple in the cases $G=\mathrm{SL}_{r+1},\mathrm{Sp}_{2r}$, but is not almost simple in the cases $G=\mathrm{SO}_{2r}^{+},\mathrm{SO}_{2r+1}$. If they were, then $G(\mathbb{F}_q)/Z(G(\mathbb{F}_q))$ would be simple as an abstract group, as shown for instance in \cite[\S 29.5]{Hum95a}; however, when $q$ is odd the quotient above has proper simple normal subgroups $\mathrm{P}\Omega_{2r}^{+}(q)$ and $\mathrm{P}\Omega_{2r+1}(q)$ (see \cite[\S 3.7.1]{Wil09}). In this paper we study the orthogonal groups and sacrifice almost simplicity rather than focusing on their simple subgroups, because the latter ones cannot be the $\mathbb{F}_{q}$-points of a variety of degree independent from $q$, or else we would violate Proposition~\ref{pr:lwnotq}.
\end{enumerate}
\end{remark}

We have talked above about a few properties of the untwisted classical groups $G$, such as connectedness, almost simplicity, and the maximal degree $\iota$ as defined in \eqref{eq:notation}. Let us list some other useful facts. Their dimensions $\delta$ are
\begin{align*}
\dim(\mathrm{SL}_{r+1}) & =r^{2}+2r, & \dim(\mathrm{SO}_{2r}^{+}) & =2r^{2}-r, & \dim(\mathrm{SO}_{2r+1})=\dim(\mathrm{Sp}_{2r}) & =2r^{2}+r.
\end{align*}
If we work in the finite field $\mathbb{F}_{q}$, their sizes are given in~\cite[\S 3.3.1, \S 3.5, \S 3.7.2, \S 3.8.2]{Wil09}. For us, it suffices to know that $|G(\mathbb{F}_{q})|\geq q^{\delta-r}(q-1)^{r}$ (see Proposition~\ref{pr:lwlowg}) and that $|G(\mathbb{F}_{q})|\leq 2q^{\delta}$.

We will also work with their Lie algebras. For $G\leq\mathrm{GL}_{N}$ as above, the corresponding Lie algebra $\mathfrak{g}\subseteq\mathrm{Mat}_{N}$ is a variety with $\dim(\mathfrak{g})=\delta$, $\deg(\mathfrak{g})=1$, and $|\mathfrak{g}(\mathbb{F}_{q})|=q^{\delta}$.

\subsection{$K$-points outside varieties}\label{subs:pointsoutV}

Knowing that we are dealing with a proper subvariety is not always enough: if $K$ is not algebraically closed, we may have $V(K)=W(K)$ even if $V\subsetneq W$. Thus, we need a result that ensures the existence of a $K$-point outside a proper subvariety of $G$. The restriction on $|K|$ in the main theorem originates here.

Let us first consider the case of $K$ infinite.

\begin{lemma}\label{lem:kperf}
Let $G\leq\mathrm{GL}_{n}$ be a connected reductive linear algebraic group over an infinite field $K$. Let $V\subseteq\mathrm{Mat}_{n}$ be a variety defined over $\overline{K}$. If $G\not\subseteq V$, then $G(K)\not\subseteq V(K)$.
\end{lemma}

\begin{proof}
If $G$ is connected and reductive, then it is {\em unirational} \cite[Thm.~18.2]{Bor91}. We call a morphism $f:X\to Y$ {\em dominant} if $Y = \overline{f(X)}$. Unirationality implies that there is a dominant morphism $f:U\rightarrow G$, where $U$ is a nonempty open subset of the affine space $\mathbb{A}^{m}$ for some $m$, with $U,f$ both defined over $K$ \cite[\S AG.13.4, AG.13.7]{Bor91}. 

Since $K$ is infinite, we know that $\mathbb{A}^{m}(K)$ is dense in $\mathbb{A}^{m}$.\footnote{This can be shown easily by  induction on $m$.
Let $V\subseteq \mathbb{A}^m$ contain $\mathbb{A}^m(K)$.
Then, for every $x_m\in K$, the intersection $V\cap (\mathbb{A}^{m-1}\times \{x_m\})$ contains every point of $\mathbb{A}^{m-1}(K)\times \{x_m\}$, and thus, by the inductive hypothesis, it must equal $\mathbb{A}^{m-1}\times \{x_m\}$. However, $V$ cannot contain the infinitely many $(m-1)$-dimensional varieties $\mathbb{A}^{m-1}\times \{x_n\}$ unless $V$ is $m$-dimensional.
Hence, $V = \mathbb{A}^m$.}
Since $\mathbb{A}^{m}(K)$ is dense in $\mathbb{A}^{m}$, so is $U(K)$. Thus, the image of $U(K)$ via the dominant map $f$ is also dense in $G$, and it is contained in $G(K)$ because $f$ is defined over $K$. Therefore $G(K)$ is dense in $G$, which means that no proper subvariety of $G$ can contain the whole $G(K)$.
\end{proof}

If $K=\mathbb{F}_{q}$ is a finite field, we need to provide instead upper and lower bounds on the number of points on a variety: if there is a gap between the number of points on $G$ and on a proper subvariety thereof, we are done.

Providing an upper bound on the number of points with coordinates in $\mathbb{F}_q$ on a hypersurface, say, is very easy: the problem just reduces to the fact that a polynomial equation of degree $n$ has at most $n$ solutions. One can then generalize this result in non-trivial ways. One generalization is the familiar Schwartz-Zippel or DeMillo-Lipton-Schwartz-Zippel lemma \cite{Sch80}; there, the coordinates are constrained to lie in a subset of $\mathbb{F}_q$. 

We will need a different generalization, in two distinct senses. First, we want to bound the number of points on an arbitrary variety $V$, not just on a hypersurface. Second, we do not want to require $V$ to be defined over $\mathbb{F}_q$; we want to allow $V$ to be defined over an extension of $\mathbb{F}_q$, while still counting only the points with coordinates over $\mathbb{F}_q$. Fortunately, there is just such a generalization in the literature.

\begin{proposition}\label{pr:lwnotq}\cite[Prop.~2.3]{LR15}
Let $V\subseteq\mathbb{A}^{m}$ be a variety defined over $\overline{\mathbb{F}_{q}}$, of dimension $d$ and degree $D$. Then $|V(\overline{\mathbb{F}_{q}})\cap\mathbb{A}^{m}(\mathbb{F}_{q})|\leq Dq^{d}$.
\end{proposition}

For a lower bound for general varieties, one can rely on explicit variants of the Lang-Weil bounds~\cite{LW54}, for instance the results of Cafure and Matera \cite{CM06}. However, for connected reductive groups there exist stronger bounds on the number of $\mathbb{F}_{q}$-points. There are exact expressions for $|G(\mathbb{F}_{q})|$ in the literature, as in \cite[p.~75]{Car93} (or \cite[Thm.~9.4.10]{Car89} for Chevalley groups). What we give here is a proof of a lower bound that is close to the actual value and as self-contained as possible.

\begin{proposition}\label{pr:lwlowg}
Let $G$ be a connected reductive linear algebraic group of rank $r$ over $\mathbb{F}_{q}$. Then
\begin{equation*}
|G(\mathbb{F}_{q})|\geq q^{\dim(G)-r}(q-1)^{r}.
\end{equation*}
\end{proposition}

\begin{proof}
For any connected $G$, a {\em Borel subgroup} is a maximal connected solvable subgroup of $G$ \cite[\S 11.1]{Bor91}.  Any connected $G$ defined over $\mathbb{F}_{q}$ has a Borel subgroup $B$ defined over $\mathbb{F}_{q}$ \cite[Prop.~16.6]{Bor91}. Since $B$ is connected and defined over $\mathbb{F}_{q}$, it has a maximal torus $T$ defined over $\mathbb{F}_{q}$ \cite[Prop.~16.6]{Bor91}, and $T$ is also a maximal torus of $G$ itself \cite[Cor.~11.3(1)]{Bor91}. Since $B$ is connected, solvable and defined over $\mathbb{F}_{q}$, its unipotent part $U$ is connected and $\mathbb{F}_{q}$-closed \cite[Thm.~10.6(1)]{Bor91}, and since $U$ is $\mathbb{F}_{q}$-closed it is defined over $\mathbb{F}_{q}$, because $\mathbb{F}_{q}$ is perfect \cite[AG.12.2]{Bor91}. For $B,T,U$ as above we have $B=T\ltimes U$ \cite[Thm.~10.6(4)]{Bor91}, and the map $T\times U\rightarrow T\ltimes U$ that multiplies the matrices while ignoring the group structure is an isomorphism of varieties, trivially defined over $\mathbb{F}_{q}$.

If $G$ is also reductive, for fixed $B,T$ as above there is a unique {\em opposite} Borel subgroup $B'$. Two equivalent characterizations of $B'$ are: the unique $B'$ for which $B\cap B'=T$, and, if $\Phi=\Phi(T,B)$ is the root system of $B$ with respect to $T$, the $B'$ whose root system is $-\Phi$ \cite[Thm.~14.1]{Bor91}. For either $B$ or $B'$, both of which are normalized by $T$ since they both contain it, being defined over $\mathbb{F}_{q}$ is equivalent to their root systems being invariant under the action of $\mathrm{Gal}(\overline{\mathbb{F}_{q}}/\mathbb{F}_{q})$ \cite[Lem.~20.3]{Bor91}; therefore, since the two root systems are opposite to each other (and one is invariant if and only if the other is), $B'$ is defined over $\mathbb{F}_{q}$ too. Arguing as for $B$, we see that the unipotent part $U'$ of $B'$ is connected and defined over $\mathbb{F}_{q}$. The product map $B\times U'$ is an isomorphism onto an open subset of $G$ \cite[Cor.~14.14]{Bor91}; again, the product map is trivially defined over $\mathbb{F}_{q}$. Putting together the two maps, there is then an injective map $\phi:U\times T\times U'\rightarrow G$ defined over $\mathbb{F}_{q}$. Since $\phi,U,U',T$ are all defined over $\mathbb{F}_{q}$, we obtain $\phi(U(\mathbb{F}_{q})\times T(\mathbb{F}_{q})\times U'(\mathbb{F}_{q}))\subseteq G(\mathbb{F}_{q})$, and since $\phi$ is injective we conclude that $|G(\mathbb{F}_{q})|\geq|U(\mathbb{F}_{q})||T(\mathbb{F}_{q})||U'(\mathbb{F}_{q})|$. Hence, we have reduced the problem of counting points to the connected unipotent and torus cases.

For $U$ connected unipotent defined over $\mathbb{F}_{q}$, $|U(\mathbb{F}_{q})|=q^{\dim(U)}$ \cite[I.1.4]{Oes84}. For $T$ a torus defined over $\mathbb{F}_{q}$, $|T(\mathbb{F}_{q})|=\det(q-\psi)$ where $\psi$ is the action of the Frobenius automorphism on the character group of $T$ \cite[I.1.5]{Oes84}. The torus splits in some finite extension $\mathbb{F}_{q^{m}}$ of $\mathbb{F}_{q}$, so up to conjugation its elements are diagonal with entries in $\mathbb{F}_{q^{m}}$. 

For all $g\in T$ and all characters $\chi\in\hat{T}$, the action of $\psi$ is defined as $\psi(\chi)(g)=\chi(g^{q})$.
Thus, $\psi^{m}$ acts trivially on all characters of $T$, because $x^{q^{m}}=x$ for all $x\in\mathbb{F}_{q^{m}}$. Hence, the eigenvalues $\lambda_{i}$ of $\psi$ are all $m$-th roots of unity, and so $|T(\mathbb{F}_{q})|=\prod_{i}(q-\lambda_{i})\geq(q-1)^{\dim(T)}$.
Since $\phi$ is an isomorphism onto an open set of $G$, which is also dense because $G$ is irreducible, we have $\dim(U_{1})+\dim(U_{2})+\dim(T)=\dim(G)$. By definition, $r=\dim(T)$. The result follows.
\end{proof}

Putting together the two bounds, we obtain the following.

\begin{corollary}\label{co:langweil}
Let $G\leq\mathrm{GL}_{n}$ be a connected reductive linear algebraic group of rank $r$ with $\dim(G)=\delta$ and $\deg(G)=\Delta$ over $\mathbb{F}_{q}$, and let $V\subseteq\mathrm{Mat}_{n}$ be a variety of degree $D$ defined over $\overline{\mathbb{F}_{q}}$. If $q>r+\Delta D$ and $G\not\subseteq V$, then $G(\mathbb{F}_{q})\not\subseteq V(\overline{\mathbb{F}_{q}})$.
\end{corollary}

\begin{proof}
The intersection $V\cap G$ is a variety of dimension $\leq\delta-1$, since it is properly contained in $G$ irreducible, and by B\'ezout it has degree $\leq\Delta D$. Therefore $V(\overline{\mathbb{F}_{q}})\cap\mathbb{A}^{m}(\mathbb{F}_{q}) \leq\Delta D q^{\delta-1}$ by Proposition~\ref{pr:lwnotq}. Using Proposition~\ref{pr:lwlowg} on $G$ we obtain
\begin{align*}
|G(\mathbb{F}_{q})\setminus (G(\mathbb{F}_q)\cap V(\overline{\mathbb{F}_{q}}))| & \geq q^{\delta}\left(1-\frac{1}{q}\right)^{r}-\Delta Dq^{\delta-1}\geq q^{\delta}\left(1-\frac{r+\Delta D}{q}\right).
\end{align*}
By our choice of $q$, the bound above is positive.
\end{proof}

With a little more work, when $G$ is an untwisted classical group, our condition on $q$ can be improved from what we would have if we just used Corollary~\ref{co:langweil} as it is. 
We will nearly eliminate the factor of $\Delta = \deg(G)$ (which can be 
rather large), in part by using a standard trick that originated with Cayley.

\begin{corollary}\label{co:lwchev}
Let $G\leq\mathrm{GL}_{N}$ be an untwisted classical group of rank $r$ over $\mathbb{F}_{q}$, and let $V\subseteq\mathrm{Mat}_{N}$ be a variety defined by a single equation $F=0$, where $F$ is a polynomial of degree $D$ over $\overline{\mathbb{F}_{q}}$ that does not vanish identically on $G$. If $q\geq 5 (r+1)^{2}D$, then $G(\mathbb{F}_{q})\setminus(G(\mathbb{F}_q)\cap V(\overline{\mathbb{F}_{q}}))\neq\emptyset$.
\end{corollary}

The proof is the same as in \cite{BDH21}. ``Untwisted classical group'' is as in Definition~\ref{de:untwisted-classical}.

\begin{proof}
Let $\dim(G)=\delta$. For the lower bound we can use $|G(\mathbb{F}_{q})|\geq q^{\delta}-rq^{\delta-1}$, by Proposition~\ref{pr:lwlowg}. For the upper bound on $|G(\mathbb{F}_{q})\cap V(\overline{\mathbb{F}_{q}})|$ we follow two routes, according to whether $G=\mathrm{SL}_{r+1}$ or $G\neq\mathrm{SL}_{r+1}$.

First, let $G=\mathrm{SL}_{r+1}$. Recall that the embedding used is the product of the usual and the contragredient one. We want to reduce the problem to the usual definition of $\mathrm{SL}_{r+1}$ as the variety $\{g\in \mathrm{Mat}_{r+1}: \det(g)=1\}$, taking advantage of the fact that $\deg(G)$ is much smaller in that case. If we start with the embedding of $\SL_{r+1}$ in 
$\mathrm{Mat}_{2 r+2}$, we project it down to the usual embedding in $\mathrm{Mat}_{r+1}$ (the one in the upper-left corner).
As a result, our variety $V$ becomes a variety $\tilde{V}=\{\tilde{F}=0\}$ with $\deg(\tilde{F})\leq rD$: in fact, $\tilde{F}$ can be obtained by replacing every coordinate of $F$ corresponding to an entry of the inverse matrix with the appropriate minor, whose degree is $r$ in the entries of the matrix. 

For the usual embedding we have $\deg(\mathrm{SL}_{r+1})\leq r+1$, and so $\deg(\SL_{r+1}\cap V)\leq r D \cdot (r+1)$. Then, using Proposition~\ref{pr:lwnotq}, we see that
\begin{equation*}
|G(\mathbb{F}_{q})\setminus(G(\mathbb{F}_q)\cap V(\overline{\mathbb{F}_{q}}))|\geq q^{\delta}\left(1-\frac{r}{q}-\frac{r D\cdot (r+1)}{q}\right).
\end{equation*}
For our $q$, this bound is positive, since $r + r (r+1) D < (r+1)^2 D$.

Now let $G\neq\mathrm{SL}_{r+1}$. The improvement here will come from working in the Lie algebra $\mathfrak{g}$, which has degree $1$. For each of the algebraic groups we are considering, 
there is an invertible birational map between dense open sets $\lambda:\mathfrak{g}\setminus(\mathfrak{g}\cap Z)\rightarrow G\setminus(G\cap Z)$ given by $\lambda(x)=(\mathrm{Id}_{N}-x)(\mathrm{Id}_{N}+x)^{-1}$, where $Z=\{x\in\mathrm{Mat}_{N}:\det(\mathrm{Id}_{N}+x)=0\}$. This construction is due to Cayley
for $G = \mathrm{SO}_n$ (with the common definition $\SO_n= \{x\in \mathrm{Mat}_n: 
\det(x)=1, x^\top x = \mathrm{Id}_{n}\}$); Weil later noted it generalizes to groups
with an involution, including $\Sp_{2r}$
(see, e.g., \cite[Ex.~1.16]{LPR06}). It is easy to verify that Cayley's construction
applies verbatim to $G=\mathrm{SO}_{2r}^{+}$ and $G=\mathrm{SO}_{2r+1}$ defined as in \S\ref{se:chev}; these cases are also covered by Weil's generalization. We will work in $\mathrm{Mat}_{N}\times\mathbb{A}^{1}$, so as to make $\lambda$ into a polynomial map, and thus an isomorphism. 

Let $G\leq\mathrm{GL}_{N}\leq\mathrm{Mat}_{N}$, and let
\begin{equation*}
X=\{(x,y)\in\mathrm{Mat}_{N}\times\mathbb{A}^{1}\,:\,x\in\mathfrak{g},\,\det(\mathrm{Id}_{N}+x)y=1\}.
\end{equation*}
Then $X$ is isomorphic to $\mathfrak{g}\setminus(\mathfrak{g}\cap Z)$ via the projection $\pi(x,y)=x$, and to $G\setminus(G\cap Z)$ via
\begin{align*}
\tilde{\lambda} & :\mathrm{Mat}_{N}\times\mathbb{A}^{1}\rightarrow\mathrm{Mat}_{N}, & \tilde{\lambda}(x,y) & =y(\mathrm{Id}_{N}-x)\mathrm{adj}(\mathrm{Id}_{N}+x).
\end{align*}
The variety $\tilde{V}=\tilde{\lambda}^{-1}(V)$ is defined by the single polynomial $\tilde{F}(x,y)=F(\tilde{\lambda}(x,y))$, which has degree $\leq \deg(\lambda)\cdot \deg(V) = (N+1) D$; moreover, if $X(\mathbb{F}_{q})\setminus(X(\mathbb{F}_{q})\cap\tilde{V}(\overline{\mathbb{F}_{q}}))$ is nonempty then $G(\mathbb{F}_{q})\setminus(G(\mathbb{F}_{q})\cap V(\overline{\mathbb{F}_{q}}))$ is also nonempty. The sets $X(\mathbb{F}_{q})$ and $\mathfrak{g}(\mathbb{F}_{q})\setminus(\mathfrak{g}\cap Z)(\mathbb{F}_{q})$ are in bijection with each other via $\pi$, so
\begin{equation*}
|X(\mathbb{F}_{q})|=|\mathfrak{g}(\mathbb{F}_{q})|-|(\mathfrak{g}\cap Z)(\mathbb{F}_{q})|\geq q^{\delta}-q^{\delta-1}N,
\end{equation*}
by $\dim(\mathfrak{g})=\delta$, B\'ezout, and Proposition~\ref{pr:lwnotq}. By $\deg(X)\leq N+1$ and Proposition~\ref{pr:lwnotq},
\begin{equation}\label{eq:wetleg}
|X(\mathbb{F}_{q})\setminus(X(\mathbb{F}_{q})\cap\tilde{V}(\overline{\mathbb{F}_{q}}))|\geq q^{\delta}\left(1-\frac{N}{q}-\frac{(N+1)^{2}D}{q}\right).
\end{equation}
Since $N\leq 2 r + 1$, we see that
$N + (N+1)^2 D \leq 
(2 r + 1) + 4 (r+1)^2 D < 5 (r+1)^2 D$, and so the lower bound in 
\eqref{eq:wetleg} is positive for $q\geq 5 (r+1)^2 D$.
\end{proof}

For technical reasons we shall need a stronger conclusion as well.

\begin{corollary}\label{co:lw4}
Let $G\leq\mathrm{GL}_{n}$ be a connected reductive linear algebraic group of rank $r$ over $\mathbb{F}_{q}$. Let $V\subsetneq G$ be a proper subvariety of degree $D$ defined over $\overline{\mathbb{F}_{q}}$. If $q>r+4D$, then $|G(\mathbb{F}_q)\cap V(\overline{\mathbb{F}_{q}})|<|G(\mathbb{F}_{q})|/4$.
\end{corollary}

\begin{proof}
The proof is analogous to the one for Corollary~\ref{co:langweil}, although here $V$ is already contained in $G$.
\end{proof}

\subsection{Escape from subvarieties}\label{se:escape}

An {\em escape from subvarieties} routine allows us to find generic elements of $G(K)$ in an appropriately bounded number of steps: in other words, if $V$ is the proper subvariety expressing ``non-genericity'', for any generating set $A$ there is an element of $A^{k}$ outside $V$, where $k$ depends only on the data of $G$ and $V$.

The idea first appears in the literature in~\cite{EMO05}, although it is likely older than that. In the context of diameter bounds, \cite[Prop.~4.1]{Hel11}, \cite[Lem.~3.11]{BGT11}, and~\cite[Lem.~48]{PS16} are results of this kind.

We will use the following version, which offers a much improved quantitative estimate.

\begin{proposition}\label{pr:escape}\cite[Cor.~3.3]{BDH21}
Let $G\leq\mathrm{GL}_{n}$ be a linear algebraic group defined over a field $K$, and let $A\subseteq G(K)$ be a finite set with $e\in A$.

For any variety $V\subseteq\mathrm{Mat}_{n}$ with $d=\dim(V)$ and $D=\deg(V)$,
and for any element $x\in\mathrm{Mat}_{n}(K)$ with
$\langle A\rangle x\not\subseteq V(K)$, there is an element $g\in A^{k}$ such that $g x\not\in V(K)$ with $k=d+1$ (if $D=1$) or $k = 2 D^{d+1}$ (if $D>1$).
\end{proposition}

A proof in the style of~\cite[Prop.~3.2]{EMO05} or~\cite[Prop.~4.1]{Hel11} leads to a value of $k$ in the statement above of order comparable to $D^{2^{d}}$, rather than to $D^{d}$. Therefore, if we were to use this $k$ later in the proof of the main theorem -- in particular: inside~\eqref{eq:exitfibre} -- the final constant $C_{2}$ would become doubly exponential in $r$.

In the statement of \cite[Cor.~3.3]{BDH21}, $\langle A\rangle$ is the {\em group} generated by $A$ and we have the extra condition $A = A^{-1}$. However, it is easy to remove the condition and have $\langle A\rangle$ denote the {\em semigroup} generated by $A$ (as in the statement of Theorem~\ref{th:main}). Of course, if $K$ is finite, then $G(K)$ is finite, and the two meanings of $\langle A\rangle$ coincide.


\section{Exceptional loci}\label{se:sing}

In order to control the number of elements of a set $A$ on a variety during the main proof in \S\ref{se:dimest}, we will proceed by induction on the dimension of the variety. We use maps to reduce the original problem to that on varieties of smaller dimensions. However, this forces us to keep in check the fibres that are larger than the generic ones, because they could lead to over-counting.

Recall that by Theorem~\ref{th:chev2}\eqref{th:chev23} the points sitting in a larger-than-expected fibre form a proper subvariety. We call the subvariety in~\eqref{eq:exlocus} the {\em exceptional locus} of $X$ through the map $f$, in analogy with the exceptional locus of the connected half of the Stein factorization: in fact, when $\dim(X)=\dim(\overline{f(X)})$, the two notions coincide.
(See, e.g., \cite[Thm.~8.3.8]{Harder}
for more on the Stein factorization, though we shall not be using it.)
The components of the fibres that may cause problems in the counting are all contained in this subvariety.
Our end goal for this section is to find a subvariety of $X$ that is proper, contains the exceptional locus, and for whose degree we can give good effective bounds.

\begin{lemma}\label{le:subclaim}
Let $X\subseteq\mathbb{A}^{s}$ be a variety with $\dim(X)\geq d$, and let $\{f_{i}\}_{i\leq t}$ be a finite set of polynomials of degree $\leq D$ defining $X$. Let $\mathcal{G}=G(s-d+1,s+1)$ be the Grassmannian; let $U\subseteq\mathcal{G}$ be an open subset 
parametrizing $(s-d)$-dimensional affine subspaces
of $\mathbb{A}^s$.

Then the set $W=\{L\in U:L\cap X=\emptyset\}$ is contained in a proper subvariety of $\mathcal{G}$ defined by equations whose coefficients are polynomials on the coefficients of the polynomials $f_{i}$ (for $d$, $D$, $s$, $t$ fixed).
\end{lemma}

\begin{proof}
As we mentioned in \S\ref{sss:degree}, a variety of dimension $\geq d$ has nonempty intersection with a generic $(s-d)$-linear affine subspace $L$. Thus, there is an open set $U'\subseteq U$ with $L\cap X\neq\emptyset$ for all $L\in U'$.

Write $\vec{d}$ for an $s$-tuple of non-negative integers, and $|\vec{d}|$ for the sum of its entries. Let $x_{\vec{d}}$ be the monomial in the $s$ variables of $\mathbb{A}^{s}$ whose degrees are given by $\vec{d}$. A tuple $(l_{j,\vec{d}})_{j\leq d,|\vec{d}|\leq 1}\in\mathbb{A}^{d(s+1)}$ defines a unique affine subspace of $\mathbb{A}^{s}$ via the $d$ equations $l_{j}=0$ with $l_{j}=\sum_{|\vec{d}|\leq 1}l_{j,\vec{d}}x_{\vec{d}}$, and the tuples for which the subspace is of dimension $s-d$ form an open subset of $\mathbb{A}^{d(s+1)}$. By standard arguments then, finding a proper subvariety of $\mathcal{G}$ with the required properties reduces to finding an analogous proper subvariety of $\mathbb{A}^{d(s+1)}$.

Since $X$ is defined by $t$ equations $f_{i}=0$, similarly to above we write $(f_{i,\vec{d}})_{i\leq t,|\vec{d}|\leq D}$ for the tuple in $\mathbb{A}^{t\binom{s+D}{D}}$ defining the polynomials $f_{i}$.

By the Nullstellensatz, $L\cap X=\emptyset$ corresponds to the existence of polynomials $a_{i},b_{j}$ such that $\sum_{i}a_{i}f_{i}+\sum_{j}b_{j}l_{j}=1$. Since there exists a bound on $\deg(a_{i}),\deg(b_{j})$ depending only on $s$ and $D$ (say $s^{s+3}D^{s+1}$, see~\cite{Som99}), the previous identity corresponds to a system of linear equations $\sum_{i}\sum_{\vec{d}'}a_{i,\vec{d}'}f_{i,\vec{d}-\vec{d}'}+\sum_{j}\sum_{\vec{d}''}b_{j,\vec{d}''}l_{j,\vec{d}-\vec{d}''}=\delta_{\vec{d}}$ for a finite collection of degree vectors $\vec{d}$ independent from the values $f_{i,\vec{d}}$, where $\delta_{\vec{d}}=1$ if $\vec{d}=\vec{0}$ and $\delta_{\vec{d}}=0$ otherwise. Arranging the $a_{i,\vec{d}'},b_{j,\vec{d}''}$ in a column vector $\vec{y}$, the system becomes a matrix equation $M\vec{y}=\vec{e}_{1}$, where $\vec{e}_{1}=(1\ 0\ \ldots\ 0)^{T}$ and the entries of $M$ are linear in the $f_{i,\vec{d}'},l_{j,\vec{d}''}$. Thus, $L\cap X=\emptyset$ if and only if $M\vec{y}=\vec{e}_{1}$ has a solution, which is equivalent to asking for $\vec{e}_{1}$ to be in the span of the columns of $M$.

The matrix $M$ depends on $L$ via the coefficients $l_{j,\vec{d}}$. Consider from now on the entries of $M$ as polynomials in the variables $l_{j,\vec{d}}$. Let $V(h)$ be the variety of $L\in\mathcal{G}$ for which all the $h\times h$ minors of $M$ vanish: for every $h_{1}\leq h_{2}$, $V(h_{1})\subseteq V(h_{2})$. Given $L$, call $h_{L}$ the smallest integer such that $L\in V(h_{L})$; then, $L\cap X=\emptyset$ if and only if the $h_{L}\times h_{L}$ minors of $(M|\vec{e}_{1})$ vanish as well. If $h$ is the largest $h_{L}$ among all $L\in U$, the set of $L$ for which $h_{L}=h$ is the open subset $U''=U\setminus(U\cap V(h-1))$ of $U$. Since $\mathcal{G}$ is irreducible, $U'\cap U''$ is nonempty: thus, $m\neq 0$ for at least one $h\times h$ minor $m$ of $(M|\vec{e}_{1})$ and at least one $L\in U$.

Hence, the variety given by $m=0$ is proper in $\mathcal{G}$, and it contains $W$. Moreover, since the entries of $M$ are polynomials in the $f_{i,\vec{d}}$, so are the coefficients of $m$.
\end{proof}

As we said before, given $X$ and $f$ we want to find a subvariety of $X$ containing the exceptional locus of $X$ through $f$.

\begin{proposition}\label{pr:singproj}
Let $X\subseteq\mathbb{A}^{s}$ be an irreducible variety over a field $K$, and let $f:X\rightarrow\mathbb{A}^{t}$ be
a morphism. 
Let $u=\dim(X)-\dim(\overline{f(X)})$.
Then there is 
a proper subvariety $Z\subsetneq \overline{f(X)}$ of degree
$$\deg(Z) \leq \mdeg(f)^{\dim (\overline{f(X)}) - 1} \deg(X)$$
such that every $y\in f(X)$ with $\dim(f^{-1}(y)) > u$ lies on $Z$.
\end{proposition}

The idea of the proof is simple:
for $y$ on $\mathbb{A}^t$ such that $\dim (f^{-1}(y))>u$,
a generic affine subspace $L$ of codimension $u+1$ will intersect
$f^{-1}(y)$. At the same time, $L$ will intersect $X$ only 
in a set of dimension $\dim(X)-(u+1)$. Hence, 
$\overline{f(L\cap X)}$ is a proper subvariety
of $\overline{f(X)}$ containing every $y$ such that $\dim (f^{-1}(\{y\}))
> u$. Of course, in the actual proof,  we have to be careful about what is meant by ``generic'', especially since one $L$ will have to make do for all $y$.

\begin{proof}
By Theorem~\ref{th:chev2}\eqref{th:chev23}, the Zariski closure of
$\{y\in\overline{f(X)}(\overline{K}):\dim(f^{-1}(y))\geq u+1\}$
is a proper subvariety
$Y\subsetneq \overline{f(X)}$. 
We assume that $u<s$, as otherwise there are no
$y$ with $\dim(f^{-1}(y))>u$.

Let $\mathcal{G} = 
G(s-u,s+1)$ and $U\subseteq \mathcal{G}$ be as in Lemma~\ref{le:subclaim}, that is, let $U$ parametrize affine subspaces of 
$\mathbb{A}^s$ of codimension $u+1$. Let $\{Y_i\}_{i\in I}$ be the irreducible components of $Y$.
We apply
Lemma~\ref{le:subclaim} to varieties of the
form $f^{-1}(y)$ (instead of $X$) for $y$ on $Y_i$
with $d=u+1$. 
For each such $y$, there is a proper subvariety $W_{y}$ of $\mathcal{G}\times\{y\}$ that contains the set $\{(L,y)\in (\mathcal{G}\times\{y\})(\overline{K}) : L\cap f^{-1}(y) = \emptyset\}$. Most importantly, the equations defining $W_{y}$ depend polynomially on $y$. Thus, we obtain that there is a proper subvariety $W$ of $\mathcal{G}\times Y_i$ itself, containing the set $\{(L,y)\in (\mathcal{G}\times Y_i)(\overline{K}) : L\cap f^{-1}(y) = \emptyset\}$. 

By Theorem~\ref{th:chev2}\eqref{th:chev21} (applied to the projection
$\mathcal{G}\times Y_i \rightarrow \mathcal{G}$) there is
a nonempty open subset $U_i\subseteq\mathcal{G}$ such that, for every $L$ in 
$U_i$, the intersection $W\cap (\{L\}\times Y_i)$
is a proper subvariety of $Y_i$.
Let $U' = U\cap \bigcup_{i\in I} U_i$.
We obtain that, for any 
$L$ in $U'(\overline{K})$, 
the complement of $W\cap (\{L\} \times Y_i)$ is open in 
$\{L\} \times Y_i$ for every $i\in I$. 
In other words, the affine subspace $L$ intersects $f^{-1}(y)$ for all $y$ in a Zariski-dense subset of $Y$.

At the same time, for a generic $L$ in $U$ (that is, any $L$
in some open subset $U''$ of $U$), the intersection
$X\cap L$ has dimension $\dim (X) - (s-\dim (L)) =
\dim (X) - (u+1)$.

Choose an $L\in (U' \cap U'')(\overline{K})$. 
Then
$\overline{f(X\cap L)}$ contains $Y$. Let $Z = \overline{f(X\cap L)}$. Then
$$\dim (Z) \leq\dim(X\cap L)
= \dim (X) - (u+1) = \dim(\overline{f(X)})-1,$$
and, by Lemma~\ref{le:zarimdeg} 
and B\'ezout,
\begin{equation*}\begin{aligned}
\deg (Z)\leq
\mdeg(f)^{\dim(\overline{f(X\cap L)})}
\deg(X\cap L)\leq
\mdeg(f)^{\dim(\overline{f(X\cap L)})} \deg(X).
\end{aligned} \qedhere \end{equation*}
\end{proof}

The exceptional locus of $X$ through $f$ is contained in $f^{-1}(Z)$, which in turn is a proper subvariety of $X$ as well. Proposition~\ref{pr:singproj} and Lemma~\ref{le:invimdeg-new} yield a degree bound for $f^{-1}(Z)$, thus allowing us in \S\ref{se:dimest} to control the number of elements sitting inside large fibres.

\section{Dimensional estimates}\label{se:dimest}

In the present section, we prove Theorem~\ref{th:main}. Generally speaking, the strategy shares some features with the induction procedure of Larsen-Pink~\cite{LP11} (followed in Breuillard-Green-Tao~\cite{BGT11}), with some important differences highlighted below.

First of all, thanks to Lemma~\ref{le:skewness}, given two varieties $V,V'$ we can create a third variety $\overline{VgV'}$ of dimension larger than both, with the help of a generic element $g$ reached quickly by Proposition~\ref{pr:escape} (escape from subvarieties). To do so, we go through the Lie algebra -- and thus we are tactically closer to \cite{Hel11} than to \cite{BGT11}. Second, thanks to Lemma~\ref{le:oberred}, we know we can reduce the task of counting elements of $A$ on a variety to the task of counting elements on fibres, images, and exceptional loci of appropriate morphisms.

These two facts can be combined to set up an induction. In the main inductive step of Proposition~\ref{pr:oberstair}, we consider a composition of morphisms whose final image is $Vg_{1}Vg_{2}V\ldots g_{\ell}V$, with $\ell$ large enough that the closure of the image is $G$. All the fibres and (projections of) exceptional loci of the various morphisms have smaller dimension than $V$, so they can be dealt with by the inductive hypothesis, and $|A^k\cap G(\overline{K})|=|A^k|$, so we know how to estimate the image as well. The base case $\dim V = 0$ is trivial, so we are done.

(In comparison, the inductive procedure in \cite[Lem.~4.2]{BGT11} -- which is essentially as in Larsen-Pink \cite{LP11} -- deals with one morphism at a time, and sets up a double induction. This procedure results in quantitative results that are considerably weaker than ours.)
 
\subsection{Inductive process}

For the rest of the section, we use the notation in~\eqref{eq:notation}. We state results for linear algebraic groups under general conditions, namely connectedness and almost simplicity, and then add stronger hypotheses or conclusions for untwisted classical groups (recall that by Remark~\ref{re:untwisted}\eqref{re:untwisted-as} the orthogonal groups are not almost simple).

Let us start with the result that makes us grow in dimension.

\begin{lemma}\label{le:skewness}
Let $G\leq\mathrm{GL}_{n}$ be an almost simple linear algebraic group over a field $K$. Let $\iota = \mathrm{mdeg}(^{-1})$, the maximum degree of the inversion map. Let $V,V'$ be subvarieties of $G$ defined over $\overline{K}$, with $\dim(V)<\dim(G)$ and $\dim(V')>0$.

Then, for every $g\in G(\overline{K})$ outside a variety $W=\{x\in\mathrm{Mat}_{n}:F(x)=0\}$ with $\deg(F)\leq 1+\min\{\iota,\dim(V)\}$ and $G\not\subseteq W$, the variety $\overline{VgV'}$ has dimension $>\dim(V)$.

The same conclusion holds for $G=\mathrm{SO}_{2n}^{+},\mathrm{SO}_{2n+1}$ over $K=\mathbb{F}_{q}$, under the additional hypotheses $0<|G(K)\cap V(\overline{K})|<|G(K)|/4$ and $|G(K)\cap V'(\overline{K})|>4$.
\end{lemma}

The result, with no bound on the degree of $W$, appears in~\cite[Lem.~4.5]{LP11} (see also~\cite[Prop.~5.5.3]{Tao15}, which gives it the name ``generic skewness''). 
The proof in \cite[Lem.~4.6]{Hel19b}, which goes through Lie algebras, assumes that the Lie algebra $\mathfrak{g}$ of $G$ simple. (Over a field of positive characteristic, the Lie algebra of an almost simple linear algebraic group is not always simple; see~\cite[Cor.~2.7(a)]{Hog82} for the exceptions). We take the best of both worlds: we prove the existence of $g$ as in~\cite{LP11}, and then use Lie algebras to bound $\deg(W)$ well.

\begin{proof}
We reduce to the case of $V$, $V'$ irreducible by considering instead a component of maximal dimension from each variety. 

First, let us show that there is a $g\in G(\overline{K})$
 such that $\dim(\overline{V g V'}) > \dim (V)$.
 We can assume without loss of generality that the identity
$e$ lies in $V'$, replacing $V'$ by a translate $(v')^{-1} V'$, $v'\in V(\overline{K})$, if needed.
If $V, V'$ are irreducible, then $V\times V'$ is irreducible, and so is $\overline{VgV'}$.
If $\dim (\overline{V g V'}) = \dim(V)$, then $V g$ is a component
of $\overline{V g V'}$, and thus, since $\overline{V g V'}$ is irreducible, 
$V g = \overline{V g V'}$, and so $V g = V g V'$, i.e., $V = V g V' g^{-1}$.

The set $S=\{x\in G(\overline{K}):Vx=V\}$ is a group by definition and it is also the set of points on the variety
$V'' = \bigcap_{v\in V(\overline{K})} v^{-1} V$. 
It must be proper inside $G$, or else $V=G$ contradicting the hypothesis. 
Hence, the intersection $H = \bigcap_{g\in G(\overline{K})} g^{-1} V'' g$ is a proper normal linear algebraic
subgroup $H \lhd G$, and thus, since $G$ is almost simple, $\dim (H) = 0$. This means
there cannot be a variety $W$ of positive dimension such that every conjugate $g W g^{-1}$ is contained in $V''$, as otherwise $W$ would be contained in $H$.

However, if $V g V' g^{-1} = V$ for every $g\in G(\overline{K})$, then 
$W = V'$ would be such a variety. Contradiction. We conclude that there exists a 
$g_0\in G(\overline{K})$ such that $\dim (\overline{V g_0 V'}) > \dim (V)$.

Now we prove that all $g$ such that $\dim (\overline{V g V'}) = \dim (V)$ sit in a proper subvariety of small degree. Let $U\subseteq V$ be the dense open set 
of non-singular points on $V$. Let $U'$ be a
dense open set of non-singular points of $\overline{V g_0 V'}$
contained in $V g_0 V'$. (Such a set
exists because $V g_0 V'$ is constructible.)
Let 
$f:V\times V'\rightarrow\overline{V g_0 V'}$ be the map defined by $f(v,v')=v g_0 v'$.
We can choose
$(v,v')\in ((U\times V') \cap f^{-1}(U'))(\overline{K})$. Since we can replace $V$ by $V v_0^{-1}$ and $V'$ by $v_1^{-1} V$, we may in fact
assume that $(v,v')=(e,e)$.

Now let $\mathfrak{v}$ and $\mathfrak{v'}$ be the tangent spaces to $V$ and $V'$ at $e$.
It is clear that $\mathfrak{v}$ and $\Ad_g (\mathfrak{v'})$ are contained in the tangent
space of $\overline{V g V' g^{-1}}$. Thus, if 
$\Ad_g (\mathfrak{v'})\not\subseteq \mathfrak{v}$, we cannot have 
$V = \overline{V g V' g^{-1}}$, and so we must have $\dim (\overline{V g V'}) > \dim(V)$. 
In the other direction, we know that
 $\dim (\overline{V g_0 V'}) > \dim(V)$,
and, since $g_0$ is in an open set $U'$ contained in
$V g_0 V'$, we also know that the tangent
space to $\overline{V g_0 V'}$ at the origin
(whose dimension, incidentally, has to be at least 
$\dim (\overline{V g_0 V'})$) equals
$\mathfrak{v} + \Ad_{g_0} (\mathfrak{v}')$; hence,
$$\dim (\mathfrak{v} + \Ad_{g_0} (\mathfrak{v'}))
> \dim (V) = \dim(\mathfrak{v}).$$

Our task is to show that
\begin{equation*}
W_{1}=\{g\in G: \Ad_{g} (\mathfrak{v'}) \subseteq \mathfrak{v}\}
\end{equation*}
is contained in a closed set $W$ of small degree and not containing the whole $G$.
By the above, $g_{0}\not\in W_{1}(\overline{K})$, and thus there is a $\vec{w}\in\mathfrak{v'}$ such that $\Ad_{g_{0}}(\vec{w})\not\in\mathfrak{v}$. 

Denote $m=\dim(\mathfrak{v})$, fix a basis $\{\vec{v}_{i}\}_{i\leq m}$ of $\mathfrak{v}$, and let $M$ be the matrix of order $(m+1)\times n^{2}$ whose rows are the $\vec{v}_{i}$ and $\Ad_{g}(\vec{w})$. Since $\{\vec{v}_{i}\}_{i\leq m}$ and $\vec{w}$ are fixed, we regard the entries of the first $m$ rows of $M$ as constants, and the entries of the last row as polynomials of degree $\leq\iota+1$ on the entries of $g$ (regarded as a matrix).

Since $\Ad_{g_0}(\overline{w})\not\in \mathfrak{v}$,
there is at least one $(m+1)\times (m+1)$ minor
of $M$ that does not vanish when $g = g_0$.
Let $W\subseteq G$ be the set of zeros of that 
minor, seen as a polynomial on the entries of $g$. Then $W$ is a closed set containing $W_1$, and, since
$g_0$ does not lie on $W$, $G\not\subseteq W$. Clearly $\deg(W)\leq \iota+1$.

We repeat the same reasoning with the matrix $M'$ whose rows are the $\vec{v}_{i}g$ and $g\vec{w}$: note that these vectors span a space of dimension 
$m+1$ if and only if the vectors in $M$ do. Now the entries of $M'$ are all linear in the entries of $g$, so the variety $W$ obtained from $M'$ following the previous steps has $\deg(W)\leq m+1$. We take the $W$ with the lower degree between these two, and we are done.

Finally, let us deal with $G=\mathrm{SO}_{2n}^{+},\mathrm{SO}_{2n+1}$ for $K=\mathbb{F}_{q}$ (recall that $q$ is odd). The only fact that needs to be proved again is that $H=\bigcap_{(g,v)\in(G\times V)(\overline{K})}g^{-1}v^{-1}Vg$ cannot contain $V'$. Call $N$ the normal subgroup of index $2$ inside $G(K)$, which exists in both cases for $q$ odd (it is denoted by $\Omega_{2n}^{+}(K)$ or $\Omega_{2n+1}(K)$, depending on the case). The centre $Z(N)$ has size $\leq 2$, and $N/Z(N)$ is simple. After fixing any $v_{0}\in G(K)\cap V(\overline{K})$, we have
\begin{equation*}
|G(K)\cap H(\overline{K})|\leq|G(K)\cap v_{0}^{-1}V(\overline{K})|=|G(K)\cap V(\overline{K})|<|G(K)|/4.
\end{equation*}
Thus $(N\cap H(\overline{K}))/Z(N)$ is a normal subgroup of $N/Z(N)$ properly contained in it, so $|(N\cap H(\overline{K}))/Z(N)|=1$, giving $|G(K)\cap H(\overline{K})|\leq 4$. By hypothesis $|G(K)\cap V'(\overline{K})|>4$, so $V'\not\subseteq H$ as desired.
\end{proof}

The following will be part of the main induction process. As said above, it splits the counting of elements to fibres, images, and exceptional loci.

\begin{lemma}\label{le:oberred}
Let $X\subseteq\mathbb{A}^{s}$ be an irreducible variety over a field $K$, and let $f:X\rightarrow \mathbb{A}^{t}$ be a morphism. Write $u = \dim (X) - \dim (\overline{f(X)})$. Let $S\subseteq X(K)$ be a finite set. Then there is some $y\in f(S)$ with $\dim(f^{-1}(y))=u$ and a proper subvariety $Z$ of $\overline{f(X)}$ with $\deg(Z)\leq\mdeg(f)^{\dim(\overline{f(X)})-1}\deg(X)$ such that
\begin{equation}\label{eq:oberred}
|S\cap(X\setminus f^{-1}(Z))(\overline{K})|\leq|f(S)||S\cap f^{-1}(y)|.
\end{equation}
\end{lemma}

\begin{proof}
Apply Proposition~\ref{pr:singproj}. If $Z$ is the proper subvariety of $\overline{f(X)}$ thus obtained, then
\begin{equation*}
|S\cap(X\setminus f^{-1}(Z))(\overline{K})|\leq|f(S)|\cdot\max\{|S\cap f^{-1}(y)|:y\in f(S),y\notin Z(\overline{K})\}.
\end{equation*}
By definition of $Z$, for $y\in f(S)$ with $y\not\in Z(\overline{K})$,
the fibre $f^{-1}(y)$ is pure-dimensional of dimension $u$.
\end{proof}

For any variety or set $X$ and any $j\geq 1$, denote by $X^{\times j}$ the direct product $X\times X \times \dotsb \times X$ of $j$ copies of $X$, so as to avoid confusion with $X^{j}$ (the set of products of $j$ elements of 
a subset $X$ of a group). It is easy to see that sets contained in proper subvarieties of $V^{\ell}$ can be made to ``crumble'' into proper subvarieties of $V$, in the following sense.

\begin{lemma}\label{le:lioflov}
Let $V\subseteq \mathbb{A}^s$ be an irreducible variety defined over a field $L$. Let $Y$ be a proper subvariety of $(\mathbb{A}^s)^{\times \ell}$ not containing $V^{\times \ell}$. Let $S_1\subseteq V(L)$ be finite. Then
$$|(S_1^{\times \ell} \cap Y(\overline{L}))|\leq \ell |S_1\cap E(\overline{L})| |S_1|^{\ell-1},$$
where $E\subset \mathbb{A}^s$ is a variety not containing $V$, with degree $\leq \deg(Y)$.
\end{lemma}
\begin{proof}
Since $Y$ does not contain $V^{\times \ell}$, there must exist some $\vec{v}\in V^{\times (\ell-1)}(\overline{L})$ such that $V\times\{\vec{v}\}\not\subseteq Y$. Fix one such $\vec{v}$, and write $Y_{(1)}\times\{\vec{v}\}=Y\cap(\mathbb{A}^s\times\{\vec{v}\})$. Then $\deg(Y_{(1)})\leq\deg(Y)$, and $V\not\subseteq Y_{(1)}$. Moreover, if $\pi:V^{\times\ell}\rightarrow V$ is the projection to the first copy of $V$, $Y_{(1)}\times V^{\times(\ell-1)}$ contains every fibre $\{v_1\}\times Y_{(2)}$
of $Y$ through $\pi$
that contains $\{v_{1}\}\times V^{\times (\ell-1)}$: if $(v_1,x)\in Y(\overline{L})$ for all $x\in V^{\times(\ell-1)}(\overline{L})$, then $(v_1,\vec{v})\in Y(\overline{L})$, and so $v_1\in Y_{(1)}(\overline{L})$. Thus, we can bound
\begin{equation}\label{eq:ebreak}
|S_1^{\times\ell}\cap Y(\overline{L})|\leq|S_1\cap Y_{(1)}(\overline{L})||S_1^{\times(\ell-1)}|+|S_1| |S_1^{\times(\ell-1)}\cap Y'_{(1)}(\overline{L})|, 
\end{equation}
where $\{v_{2}\}\times Y'_{(1)}$ is a fibre of $Y$ through $\pi$ not containing $V^{\times(\ell-1)}$ such that 
$|S_1^{\times(\ell-1)}\cap Y'_{(1)}(\overline{L})|$ is maximal. 
By Bézout, $\deg(Y'_{(1)})\leq\deg(Y)$.
Now we recur, repeating the same procedure for $Y'_{(1)}$ instead of $Y$ and $\ell-1$ instead of $\ell$ to bound $|S_1^{\times (\ell-1)} \cap Y_{(1)}'(\overline{L})|$, and continue until we exhaust all copies of $V$. Finally, choose $E$ among $Y_{(1)},Y_{(2)},\ldots,Y_{(\ell-1)},Y'_{(\ell-1)}$ such that $|S_1\cap E(\overline{L})|$ is maximal.
\end{proof}

Now we come to the main inductive step. It uses Lemma~\ref{le:skewness} (generic skewness),
Proposition~\ref{pr:escape} (escape from subvarieties) and Lemmas~\ref{le:oberred}--\ref{le:lioflov}.

\begin{proposition}\label{pr:oberstair}
Let $G\leq\mathrm{GL}_{n}$ be a connected, almost simple linear algebraic group of rank $r$ with $\dim(G)=\delta$, $\deg(G)=\Delta$ and $\mathrm{mdeg}(^{-1})=\iota$, defined over a field $K$ with $|K|>r+(\iota+1)\Delta$. Let $A\subseteq G(K)$ be a finite set with $e\in A$ and $\langle A\rangle=G(K)$. Let $V\subseteq G$ be an irreducible variety defined over $\overline{K}$ with $0<\dim(V)=d<\delta$ and $\deg(V)=D$.

Then, for any positive integer $m$, there are integers $\ell\leq\delta-d+1$ and $k\leq 2 (1+\iota)^{n^{2}}$ and there are proper subvarieties 
$F_{1},\ldots,F_{\ell-1}$ of $V$ and a variety $E\subsetneq \mathrm{Mat}_n$
with
\begin{equation*}
|A^{m}\cap V(\overline{K})|^{\ell}\leq|A^{\ell m+(\ell-1)k}|\cdot\prod_{j=1}^{\ell-1}|A^{m}\cap F_{j}(\overline{K})|+\ell|A^{m}\cap E(\overline{K})||A^{m}\cap V(\overline{K})|^{\ell-1}
\end{equation*}
and $E$ not containing $V$, and satisfying the following properties as well: 
\begin{align}\label{eq:hodor}
\sum_{j=1}^{\ell-1}\dim(F_{j}) & =\ell d-\delta,
& \deg(F_{j}) & \leq 2^\delta \ell^{\delta-1} D^{j+1},
& \deg(E) &\leq (2 \ell)^{\delta} D^{\ell}.
\end{align}

If $G$ is an untwisted classical group as in \S\ref{se:chev}, the same conclusion holds under different hypotheses on $|K|$: for $G=\mathrm{SL}_{r+1},\mathrm{Sp}_{2r}$ we let $|K|\geq 10(r+1)^{2}$, and for $G=\mathrm{SO}_{2r}^{+},\mathrm{SO}_{2r+1}$ we let $|K|\geq 5(D\delta)^{\delta-1}$.
\end{proposition}

As always, $\mathrm{Mat}_n\simeq\mathbb{A}^{n^2}$ is the space of $n\times n$ matrices. Recall that, much like in Theorem~\ref{th:main}, $\langle A\rangle=G(K)$ means that $A$ generates $G(K)$ as a semigroup; we may write the same statement with $A$ generating $G(K)$ as a group, but we have to add the condition $A=A^{-1}$. If $K$ is finite then the distinction is unnecessary.

\begin{proof}
Fix an integer $\ell$ and elements $g_{1},\ldots,g_{\ell-1}\in G(\overline{K})$ (to be chosen soon), and define
\begin{align*}
V_{1} & =V, & V_{j+1} & =\overline{V_{j}g_{j}V} \ \ \ \ \ (1\leq j<\ell).
\end{align*}
For each $1\leq j\leq\ell$, let $d_{j}=\dim(V_{j})$ and $D_{j}=\deg(V_{j})$. By Lemma~\ref{le:zarimdeg}, we have $D_{j}\leq D^{j}j^{d_j}$. We start by showing that the $g_{j}$ can be chosen so that they sit in a small power of $A$ and that we have $d_{j+1}>d_{j}$ at every step.

If $G$ is a connected almost simple linear algebraic group, by Lemma~\ref{le:skewness} we get $d_{j+1}>d_{j}$ for any choice of $g_{j}$ outside a variety $W_{j}$ with $\dim(W_{j})<n^{2}$, $\deg(W_{j})\leq 1+\min\{\iota,d_{j}\}$ and $G\not\subseteq W_{j}$. Then we know by Lemma~\ref{lem:kperf}, Corollary~\ref{co:langweil}, and our condition on $|K|$ that $G(K)\not\subseteq W_{j}(\overline{K})$. Thus, since $\langle A\rangle = G(K)$, we may apply Proposition~\ref{pr:escape} (escape from subvarieties), and obtain 
\begin{align}\label{eq:stepsescape}
g_{j} & \in A^{k_{j}}, & k_{j} & \leq 2 (1+\min\{\iota,d_{j}\})^{n^{2}},
\end{align}
such that $d_{j+1}>d_{j}$. Choose such an element $g_{j}$ for each $j$, and let $\ell$ be the least index such that $V_{\ell}=G$. Clearly, $\ell\leq\delta-d+1$, and $g_{1},\ldots,g_{\ell-1}\in A^{k}$ with $k\leq 2 (1+\iota)^{n^{2}}$.

If $G$ is an untwisted classical group, we have $\iota=1$. In the cases $G=\mathrm{SL}_{r+1},\mathrm{Sp}_{2r}$, we argue as we did above but replacing Corollary~\ref{co:langweil} with Corollary~\ref{co:lwchev}: in fact, its hypothesis is satisfied because $\iota=1$ implies that the polynomial defining $W_{j}$ has degree $\leq 2$. In the cases $G=\mathrm{SO}_{2r}^{+},\mathrm{SO}_{2r+1}$, since we stop at some $\ell\leq\delta$ we have $D_{j}\leq(D\delta)^{\delta-1}$ at every step. Then the condition on $|K|$, the fact that $\delta\geq r$, and Corollary~\ref{co:lw4} imply that the hypothesis $|G(K)\cap V_{j}(\overline{K})|<|G(K)|/4$ is satisfied. On the other hand, if $|G(K)\cap V(\overline{K})|\leq 4$, then the result is trivially true by $|A|,\ell\geq 2$, whereas if $|G(K)\cap V(\overline{K})|>4$ then we obtain also $|G(K)\cap V_{j}(\overline{K})|>0$, so that all the hypotheses of Lemma~\ref{le:skewness} are satisfied. Hence, we argue as above using Lemma~\ref{le:skewness} and Corollary~\ref{co:lwchev}.

For every $G$ that we consider, we have proved that $d_{j+1}>d_{j}$ for all $j$ and for $g_{j}$ satisfying \eqref{eq:stepsescape}. Now, define the morphisms
\begin{align*}
f_{i} & :V_{i}\times V\rightarrow V_{i+1}, & & f_{i}(x,v)=xg_{i}v, \\
f_{i,i+1} & :V_{i}\times V^{\times(\ell-i)}\rightarrow V_{i+1}\times V^{\times(\ell-i-1)}, & & f_{i,i+1}=(f_{i},\mathrm{Id}_{V^{\times(\ell-i-1)}}),
\end{align*}
so that we have a chain of maps
\begin{equation}\label{eq:jcss}
   V_1\times V^{\times (\ell-1)}
   \xrightarrow{f_{1,2}} V_2\times V^{\times (\ell-2)} \xrightarrow{f_{2,3}} \dotsb
      \xrightarrow{  }
   V_{\ell-1}\times V \xrightarrow{f_{\ell-1,\ell}} V_\ell = G\end{equation}
whose composition is just the map from $V^{\times \ell} = V_1\times V^{\times (\ell-1)}$
to $G$ given by
$$(v_1,v_2,\dotsc,v_\ell) \mapsto v_1 g_1 v_2 \dotsb g_{\ell-1} v_\ell.$$ 
In general, for $1\leq i < j\leq \ell$, we can define 
$${f}_{i,j}={f}_{j-1,j}\circ{f}_{j-2,j-1}\circ\ldots\circ{f}_{i,i+1}.$$

Our plan is to apply Lemma~\ref{le:oberred} to
$f_{i,i+1}$ for each value $i=1,2,\dotsc,\ell-1$ in succession so as to bound
$|A^m\cap V(\overline{K})|^\ell$ in terms of fibres $f_{i,i+1}^{-1}(y)$ (which
will be projected injectively to copies of $V$) and the image
$f_{1,\ell-1}((A^m\cap V(\overline{K}))^{\times \ell})\subseteq A^{\ell m + (\ell-1) k}$.
We will of course have to avoid exceptional loci and keep track of various degrees
and dimensions.

(Notice it would not do to simply apply Lemma~\ref{le:oberred} once to
$f = f_{1,\ell-1}$: we would have difficulty ``crumbling'' $f^{-1}(y)$ down into
subvarieties of $V$ of the right dimensions.)

Since $V$ is irreducible, the
varieties $V_i\times V^{\times (\ell-i)}$ in \eqref{eq:jcss} are all
  irreducible. (Recall we defined each $V_i$ as a closure.)
Write $m_{i}=im+(i-1)k$ and $S_{i}=A^{m_{i}}\cap V_{i}(\overline{K})$, observing that by definition $f_{i}(S_{i}\times S_{1})\subseteq S_{i+1}$.
Obviously $\mathrm{mdeg}(f_{j})=2$.

Apply Lemma~\ref{le:oberred} with $f=f_{i}$, $X = V_i\times V$ and
$S = S_{i}\times S_{1}$. We can take the subvariety $Z_{i}\subsetneq V_{i+1}$ thus obtained and pull it
back to $V^{\times\ell}$.
It is actually more efficient to do a pull-back to all of affine space, in the following sense. We know that $V$ (and $V_i$) lives inside $G\subseteq \GL_n
\subseteq \mathrm{Mat}_n = \mathbb{A}^{s}$ for $s=n^2$. We define
\[\tilde{f}_{1,j}:(\mathbb{A}^{s})^{\times \ell} \to 
\mathbb{A}^s \times (\mathbb{A}^{s})^{\times (\ell-j)}
\]
by matrix multiplication:
\[\tilde{f}_{1,j}(x_1,\dotsc,x_\ell) = 
(x_1 g_1 x_2 \dotsb g_{j-1} x_j,x_{j+1},x_{j+2},\dotsc,x_\ell).\]
Clearly, $\tilde{f}_{1,j}|_{V^{\times \ell}}$ is just $f_{1,j}$, and
$\mdeg (\tilde{f}_{1,j}) = \mdeg (f_{1,j}) = j$. 
By Lemma \ref{le:inhypers}, there is a hypersurface $Z_i'$ containing $Z_i$ but not $V_{i+1}$ with $\deg (Z_i')\leq \deg (Z_i)$.
Putting together the pullbacks of all $Z_i'$, we get:
\begin{equation}\label{eq:break-silly}
(A^{m}\cap V(\overline{K}))^{\times\ell}=S_{1}^{\times\ell}=(S_{1}^{\times\ell}\cap U(\overline{K}))\cup(S_{1}^{\times\ell}\cap Z(\overline{K}))
\end{equation}
with
\begin{align*}
Z & =\bigcup_{i=1}^{\ell-1} \tilde{f}_{1,i+1}^{-1}(Z_{i}'\times 
(\mathbb{A}^s)^{\times (\ell-i-1)}), & U & =V^{\times\ell}\setminus 
(Z\cap V^{\times \ell}). 
\end{align*}

We start by estimating $|S_{1}^{\times\ell}\cap Z(\overline{K})|$.
 By Lemma~\ref{le:invimdeg-silly},
$$\begin{aligned}
  \deg(\tilde{f}_{1,i+1}^{-1}(Z_{i}'\times (\mathbb{A}^s)^{\times(\ell-i-1)})) &\leq
(i+1) \deg(Z_i).\end{aligned}$$
By Lemma \ref{le:oberred}, $\deg(Z_i)\leq 2^{d_{i+1}-1} D D_i$. Hence
\begin{equation}\begin{aligned}
  \deg(Z) & 
  \leq\sum_{i=1}^{\ell-1} \ell 2^{d_{i+1}-1} D D_{i}\leq
  \sum_{i=1}^{\ell-1} \ell 2^{d_{i+1}-1} D^{i+1} i^{d_i} <
  2^{\delta}\ell^{\delta}D^{\ell}. \label{eq:bounddegz}
\end{aligned}\end{equation}
Hence, by Lemma \ref{le:lioflov},
\begin{equation}\label{eq:edone}|S_{1}^{\times\ell}\cap Z(\overline{K})| \leq\ell
|A^{m}\cap E(\overline{K})| |S_1|^{\ell-1} 
\end{equation}
where $E$ (which does not contain $V$) has degree $\leq \deg(Z) < 2^{\delta} \ell^{\delta} D^{\ell}$.

Now we estimate $|S_{1}^{\times\ell}\cap U(\overline{K})|$.
By Lemma \ref{le:oberred} and the definition of $U$, 
$$ |S_{1}^{\times\ell}\cap U(\overline{K})|
\leq |f_{1,2}(S_1^{\times \ell})| |S_1^{\times \ell}\cap f_{1,2}^{-1}(y_1)|$$
for some $y_1\in f_{1,2}(S_1^{\times \ell})$ with $\dim (f_{1,2}^{-1}(y_1)) =
\dim(V^{\times \ell})- \dim(V_2\times V^{\times (\ell-2)}) = 2 d - d_2$.
Clearly
$$f_{1,2}(S_1^{\times \ell})\subseteq (S_2\times S_1^{\times (\ell-2)})
\cap f_{1,2}(U(\overline{K})).$$
Hence, again by Lemma \ref{le:oberred} and the definition of $U$,
$$\begin{aligned}|f_{1,2}(S_1^{\times \ell})|
  &\leq |f_{2,3}(f_{1,2}(S_1^{\times \ell}))|\cdot
|(S_2\times S_1^{\times (\ell-2)})\cap f_{2,3}^{-1}(y_2)|\\
&\leq |(S_3\times S_1^{\times (\ell-3)})
\cap f_{1,3}(U(\overline{K}))|\cdot
|(S_2\times S_1^{\times (\ell-2)})\cap f_{2,3}^{-1}(y_2)|\end{aligned}$$
for some $y_2\in f_{2,3}(S_2\times S_1^{\times (\ell-2)})$ with $\dim (f_{2,3}^{-1}(y_2)) =
d_2 + d - d_3$. We iterate, and get
$$|S_{1}^{\times\ell}\cap U(\overline{K})|\leq
|S_\ell|\cdot
\prod_{i=1}^{\ell-1}
|(S_i\times S_1^{\times (\ell-i)})\cap f_{i,i+1}^{-1}(y_i)|,$$
where 
$y_i\in (V_{i+1}\times V^{\times (\ell-i-1)})(\overline{K})$
is such that $\dim (f_{i,i+1}^{-1}(y_i)) = d_i + d - d_{i+1}$.

Now, for given $y_i$, if $f_{i,i+1}(x,v_{i+1},v_{i+2},\dotsc,v_{\ell}) =
(x g_i v_{i+1},v_{i+2},\dotsc,v_{\ell})$ equals $y_i$, then $v_{i+2},\dotsc,v_{\ell}$
are determined, and, if we know $v_{i+1}$, $x$ is determined as well (since $g_i$
is fixed). Hence, the projection $\pi_2:f_{i,i+1}^{-1}(y_i)\to V$
to the second coordinate is injective. We define $F_i = \overline{\pi_2(f_{i,i+1}^{-1}(y_i))}$, and conclude that 
\begin{align}\label{eq:udone}
  |S_{1}^{\times\ell}\cap U(\overline{K})|\leq |S_\ell| \prod_{i=1}^{\ell-1}
  |S_1\cap F_i(\overline{K})|
\end{align}
and $\dim(F_i) = \dim (f_{i,i+1}^{-1}(y_i)) = d_i + d - d_{i+1}$.
Moreover,
by Lemmas~\ref{le:zarimdeg}--\ref{le:invimdeg-new},

\begin{align*}
\deg(F_i) & \leq \deg(f_{i,i+1}^{-1}(y_i)) = \deg(f_{i}^{-1}(\pi_1(y_i))) \\
 & \leq \deg(V_i) \deg(V) 2^{\dim(V_{i+1})} \leq
D^{i+1} i^{d_i} 2^{d_{i+1}}
< 2^\delta \ell^{\delta-1} D^{i+1}.
\end{align*}

We apply \eqref{eq:edone} and \eqref{eq:udone} inside \eqref{eq:break-silly} and
are done. 
\end{proof}

\subsection{Main theorem}\label{subs:mainth}

It is time to prove the main theorem. We provide a slightly more general version below, and deduce Theorem~\ref{th:main} at the end of the section. 

\begin{theorem}\label{th:ober}
Let $G\leq\mathrm{GL}_{n}$ be a connected almost simple linear algebraic group of rank $r$ with $\dim(G)=\delta$, $\deg(G)=\Delta$ and $\mathrm{mdeg}(^{-1})=\iota$, defined over a field $K$ with $|K|>r+(\iota+1)\Delta$. Let $A\subseteq G(K)$ be a finite set with $e\in A$ and $\langle A\rangle=G(K)$. Let $V$ be any subvariety of $G$. Write $d=\dim(V)$ and $D=\deg(V)$.

Then, for any positive integer $m$, we have 
\begin{equation}\label{eq:desiderio}
|A^{m}\cap V(\overline{K})|\leq C_{1}|A^{C_{2}}|^{\frac{\dim(V)}{\dim(G)}},
\end{equation}
where 
\begin{align}\label{eq:boundc1c2}
C_{1}=C_{1}(d,D) & \leq(2\delta D)^{\delta^{d}}, & C_{2}=C_{2}(m) & \leq\delta(m+2 (1+\iota)^{n^{2}}).
\end{align}

If $G$ is an untwisted classical group, the same conclusion holds without any assumption on $|K|$.
\end{theorem}

{\em Remarks.} (i) Again, we would need $A=A^{-1}$ if we assumed $A$ generating $G(K)$ only as a group (no extra condition is necessary if $K$ is finite).
(ii) We will in fact prove the stronger bound $C_1 \leq (2\delta)^{\delta^{d}} D^{\delta (\delta-1) \dotsb (\delta-d+1)}$ for $G$ connected almost simple.
(iii) ``Untwisted classical group'' has the same meaning as in \S\ref{se:chev}: namely, we exclude the usual embedding $G=\mathrm{SL}_{n}\leq\mathrm{GL}_{n}$, which is however covered by the ``connected almost simple'' case (and also by Remark~\ref{re:slnusual}).

\begin{proof}
We proceed by induction on $d$. If $d=0$, then $V(\overline{K})$ is a set of $D$ points, and therefore we have $|A^{m}\cap V(\overline{K})|\leq D$. Furthermore, if $d=\delta$ then $V=G$ and so $|A^{m}\cap V(\overline{K})|=|A^{m}|$.

Now, let us fix $0<d<\delta$, and assume that for all subvarieties of dimension $<d$ the bound holds for functions $C_{1},C_{2}$ as in \eqref{eq:boundc1c2}. It is enough to consider $V$ irreducible since our $C_{1}$ will be more than linear in $D$. Apply Proposition~\ref{pr:oberstair}. Thus, we have integers $\ell\leq\delta-d+1\leq\delta$ and $k\leq 2 (1+\iota)^{n^{2}}$ and subvarieties $F_{1},\ldots,F_{\ell-1}\subsetneq V$, $E\subsetneq \mathrm{Mat}_n$ satisfying its conclusion. Write \begin{align*}
\dim(E\cap V) & =d'\leq d-1, & \deg(E\cap V) & =D'\leq (2 \ell)^{\delta} D^{\ell+1}, \\
\dim(F_{i}) & =d_{i}\leq d-1, & \deg(F_{i}) & =D_{i}\leq 2^\delta \ell^{\delta-1} D^{i+1},
\end{align*}
Using the inductive hypothesis, we have the bound
\begin{align}
|A^{m}\cap V(\overline{K})|^{\ell}\leq \ & |A^{\ell m+(\ell-1)k}|\cdot\prod_{i=1}^{\ell-1}C_{1}(d_{i},D_{i})|A^{C_{2}(m)}|^{\frac{d_{i}}{\delta}} \nonumber \\
\ & +\ell C_{1}(d',D')|A^{C_{2}(m)}|^{\frac{d'}{\delta}}|A^{m}\cap V(\overline{K})|^{\ell-1}. \label{eq:boundind}
\end{align}

Assume for now that $|A^m\cap V(\overline{K})|\geq 2 \ell C_1(d',D') |A^{C_2(m)}|^{d'/\delta}$.
Then the second summand on the right side of \eqref{eq:boundind} is 
$\leq \frac{1}{2} |A^m\cap V(\overline{K})|^\ell$.
So, by $\sum_{i=1}^{\ell-1}d_{i}=\ell d-\delta$, 
\begin{equation*}
|A^{m}\cap V(\overline{K})|^{\ell}\leq 2\prod_{i=1}^{\ell-1}C_{1}(d_{i},D_{i})\cdot|A^{\max\{\ell m+(\ell-1)k,C_{2}(m)\}}|^{1+\frac{\ell d-\delta}{\delta}}.
\end{equation*}
Hence, taking the $\ell$-th root on both sides, we obtain the main result provided that
\begin{align}\label{eq:exitfibre}
C_{1}(d,D) & \geq \left(2 \prod_{i=1}^{\ell-1} C_1(d-1,2^\delta \ell^{\delta-1}
D^{i+1})\right)^{1/\ell},
& C_{2}(m) & \geq\ell m+(\ell-1)k
\end{align}
for $1\leq d\leq \delta$,
assuming as well that $C_1(d,D)$ is increasing on $d$ for $0\leq d<\delta$.

Assume now that $|A^m\cap V(\overline{K})|< 2 \ell C_1(d',D') |A^{C_2(m)}|^{d'/\delta}$. Putting this bound together with the trivial bound $|A^m\cap V(\overline{K})|\leq |A^m|$, we see that
$$|A^m\cap V(\overline{K})|< (2 \ell C_1(d',D') |A^{C_2(m)}|^{d'/\delta})^r
|A^m|^{1-r}$$
for any $0\leq r\leq 1$. Recall $d'\leq d-1$.
We let $r = (\delta-d)/(\delta+1-d)$, and get
$$|A^m\cap V(\overline{K})|< (2 \ell C_1(d',D'))^r
|A^{C_2(m)}|^{d/\delta}$$
provided that $C_2(m)\geq m$. Thus, we are fine if
\begin{align}\label{eq:exitexc}
C_1(d,D) &\geq (2 \ell C_1(d-1,(2\ell)^\delta D^{\ell+1}))^{\frac{\delta-d}{\delta+1-d}}, & C_2(m) &\geq m.
\end{align}

It remains to choose $C_{1},C_{2}$ so that they satisfy \eqref{eq:exitfibre} and \eqref{eq:exitexc}. We take
$$C_2(m) = \delta (m+k) \leq \delta (m + 2 (1 + \iota)^{n^2}).$$
Define $C_1(0,D) = D$, and, for $d\geq 1$,
$$C_1(d,D) = (2\delta)^{\delta^d} D^{\delta (\delta-1) \dotsb (\delta-d+1)}.$$
To verify \eqref{eq:exitexc}, we note, first, that
\[C_1(1,D) = (2\delta D)^\delta >
\left((2 \delta)^{\delta+1} D^{\delta+1}\right)^{\frac{\delta-1}{\delta}}
\geq (2 \ell C_1(0,(2\ell)^\delta D^{\ell+1}))^{\frac{\delta-1}{\delta}},
\]
whereas, for $d\geq 2$,
\[\begin{aligned} C_1(d-1,(2\ell)^\delta D^{\ell+1})
&\leq (2\delta)^{\delta^{d-1}} ((2\delta)^\delta D^{\delta-d+2})^{\delta (\delta-1) \dotsb (\delta-d+2)}\\
&\leq (2\delta)^{\delta^d + \delta^{d-1}} D^{\delta (\delta-1) \dotsb (\delta-d+2)\cdot (\delta-d+2)},
\end{aligned}\]
and so, since $(\delta-d+2) (\delta-d)/(\delta+1-d) < \delta+1-d$
and $(\delta^d + \delta^{d-1}+ 1) (\delta-d)/(\delta+1-d)
< (\delta^d + \delta^{d-1} + 1) (\delta-1)/\delta < \delta^d$ for $d\geq 2$,
$$C_1(d,D) \leq \left(2 \ell C_1(d-1,(2\ell)^\delta D^{\ell+1})\right)^{\frac{\delta-d}{\delta+1-d}}.$$
Thus, \eqref{eq:exitexc} holds for all $d\geq 1$. Hence, so does
\eqref{eq:exitfibre}, which is weaker.

Lastly, let us show that we do not need the assumption on $|K|$ if $G$ is an untwisted classical group as in \S \ref{se:chev}.
We already know from Proposition \ref{pr:oberstair}  that it is enough to assume
$|K|>5(D\delta)^{\delta-1}=\max\{5(D\delta)^{\delta-1}, 10 (r+1)^2\}$. Suppose $|K|\leq 5(D\delta)^{\delta-1}$. Then, by Proposition \ref{pr:lwnotq},
$$|A^m\cap V(\overline{K})| \leq |\mathrm{Mat}_n(K)\cap V(\overline{K})|
\leq D |K|^d \leq 5^{d}D(\delta D)^{(\delta-1)d}.
$$
We know that $\delta\geq 3$, so $2^{\delta^{d}}\geq 8^{d}$. Hence
\begin{equation*}
(2\delta D)^{\delta^d} \geq 8^{d}(\delta D)^{\delta^d} > 5^{d}D(\delta D)^{(\delta-1)d} \geq|A^m\cap V(\overline{K})|. \qedhere
\end{equation*}
\end{proof}

Theorem ~\ref{th:main} is just a somewhat simplified rephrasing of
Theorem~\ref{th:ober}.
\begin{proof}[Proof of Theorem~\ref{th:main}]
As mentioned in \S\ref{se:linalg}, since $G$ is almost simple we have $G\leq\mathrm{SL}_{n}$, so $\dim(G)\leq n^{2}-1$ and the inversion map has degree $\iota\leq n-1$.
Apply Theorem~\ref{th:ober} 
with $m=1$. Since
we can assume $n\geq 2$, 
\begin{equation*}
C_{2} \leq(n^{2}-1)(1+2 n^{n^{2}})\leq n^{n^2+3}.
\end{equation*}
If $G$ is one of the groups in \S \ref{se:chev} (with no restriction on $|K|$), then we can again apply Theorem~\ref{th:ober} as before with the additional advantage that the degree $\iota$ of the inversion map is $1$. So, Theorem~\ref{th:main} holds with
\begin{equation*}
C_{2} \leq(n^{2}-1) (1+2\cdot 2^{n^{2}}) \leq n^2 2^{n^2+1}. \qedhere
\end{equation*}
\end{proof}

\begin{remark}\label{re:slnusual}
We can reduce the case of $G$ equal to the algebraic group given by $\det(x)=1$ inside $\mathrm{Mat}_{n}$ (i.e., the more natural definition of $\mathrm{SL}_{n}$)
to the case of $\SL_n$ as defined in \S \ref{se:chev}. There is a trade-off, in that
$C_2$ will be smaller and $C_1$ will be larger than if we applied Theorem \ref{th:main} directly (as we may indeed do).
Call $\varphi$ the rational map defined by
\begin{equation*}
\varphi(x)=\begin{pmatrix} x & 0 \\ 0 & (x^{-1})^{\top} \end{pmatrix},
\end{equation*}
which is an isomorphism from $G$ to the variety $\mathrm{SL}_{n}$ defined as in \S\ref{se:chev}: then, \eqref{eq:maintheq} holds for $A,V,G$ if and only if it holds for $\varphi(A),\varphi(V)\cap\mathrm{SL}_{n},\mathrm{SL}_{n}$. On the other hand, $\varphi(V)\cap\mathrm{SL}_{n}\subseteq\mathrm{Mat}_{2n}$ can be defined as a variety as follows:
\begin{equation*}
\varphi(V)\cap\mathrm{SL}_{n}=\left\{\begin{pmatrix} x_{11} & x_{12} \\ x_{21} & x_{22} \end{pmatrix}\in\mathrm{Mat}_{2n}:x_{11}\in V,x_{12}=x_{21}=0,x_{11}x_{22}^{\top}=\mathrm{Id}_{n}\right\}.
\end{equation*}
Then, by Bézout, $\deg(\varphi(V)\cap\mathrm{SL}_{n})\leq 2^{n^{2}}\deg(V)$.
Thus, by Theorem~\ref{th:ober} for $|K|>r+(\iota+1)\deg(G)=(n-1)+n^{2}$, \eqref{eq:desiderio} holds with
\begin{align*}
C_1 & \leq (2 \delta\cdot 2^{n^2} \deg(V))^{\delta^d} \leq (2^{n^2+1} n^2 \deg(V))^{(n^{2}-1)^d}, \\
C_2 & \leq (n^2-1) (m + 2^{n^2+1}),
\end{align*}
and \eqref{eq:maintheq} holds for the bounds above with $m=1$. 
If $|K|\leq(n-1)+ n^2$ we may use Proposition~\ref{pr:lwnotq} and obtain
\begin{align*}
|A^m\cap V(\overline{K})|\leq \deg(V)|K|^d<(2n^2)^{d}\deg(V),
\end{align*}
thus obtaining $C_{1}$ as above and respectively $C_{2}=m$ and $C_{2}=1$ for \eqref{eq:desiderio} and \eqref{eq:maintheq}.
\end{remark}

\section{Diameter bounds}\label{se:diambounds}

In this section, we show how to bound the diameter of untwisted classical groups $G$ over $\mathbb{F}_{q}$ using the estimate in Theorem~\ref{th:main}. The procedure is the same as in \cite{BDH21}; our aim is to present the strengths and limitations of the approach that goes through a dimensional estimate valid for general varieties.

For the whole section, the untwisted classical groups $G$ except $\mathrm{SL}_{n}$ are defined as in \S\ref{se:chev}, and the restrictions to the rank and characteristic apply here as well. The group $G=\mathrm{SL}_{n}$ is instead defined with the usual embedding $G\subseteq\mathrm{Mat}_{n}$, rather than as in \S\ref{se:chev}: clearly $\mathrm{diam}(G)$ is not affected by the choice of embedding, and we may use the estimate of Remark~\ref{re:slnusual} to replace Theorem~\ref{th:main} for $\mathrm{SL}_{n}$, so we can choose to adopt the usual embedding everywhere. We shall do so to avoid some minor technical unpleasantness related to the definition of semisimple elements.

Essentially the same arguments apply more generally to almost simple linear algebraic groups. We will be working with untwisted classical groups in part for simplicity, and in part because these diameter bounds are just an application of our main result -- an application that is not as strong as our result \eqref{eq:amago} in \cite{BDH21}.

Since we work over finite fields, the group and the semigroup generated by $A$ are the same object. Thus, all results from the previous sections hold regardless of our convention for $\langle A\rangle$. We do add the condition $A=A^{-1}$ in Theorem~\ref{th:growth} for a different reason, namely because we sometimes need to conjugate by elements of $A$.

\subsection{Preliminaries}
We begin by stating several results on conjugacy classes, regular semisimple elements, and maximal tori. 
Most of them have quantitatively stronger counterparts in \cite{BDH21}, but we provide here shorter proofs that do not qualitatively affect our conclusions.

\begin{lemma}\label{le:hth}
Let $G$ be an untwisted classical group of rank $r$ with $\dim(G)=\delta$,
defined over $\mathbb{F}_{q}$,
and let $T$ be a maximal torus in $G$. Then, the number of distinct conjugates of $T$ by elements of $G(\mathbb{F}_{q})$ is
\begin{equation*}
|\{hTh^{-1}:h\in G(\mathbb{F}_{q})\}|\geq
\frac{q^{\delta-r} (q-1)^r}{r! 2^r (q+1)^r} \geq 
\frac{q^{\delta-r}}{(6 r)^r}
.
\end{equation*}
\end{lemma}

\begin{proof}
Up to conjugation by some element of $G(\overline{\mathbb{F}_{q}})$, $T$ is defined over $\mathbb{F}_{q}$, and thus so are the normalizer $N(T)$ of $T$ in $G$ and the {\em Weyl group} $\mathcal{W}(G)\simeq N(T)/T$ (see \cite[\S 11.19]{Bor91} for the Weyl group, and \cite[\S 6.3]{Bor91} more generally for quotients of varieties). We have $|\{hTh^{-1}:h\in G(\mathbb{F}_{q})\}|=|G(\mathbb{F}_{q})|/|N(T)(\mathbb{F}_{q})|$ by definition, $|G(\mathbb{F}_{q})|\geq q^{\delta-r}(q-1)^{r}$ by Proposition~\ref{pr:lwlowg}, and $|T(\mathbb{F}_{q})|\leq(q+1)^{r}$ by the proof of that same result. Thus, it remains to bound $|N(T)(\mathbb{F}_{q})|/|T(\mathbb{F}_{q})|$.

For untwisted classical groups, $\mathcal{W}(G)$ is a finite group of size at most $r!2^{r}$ (see for instance \cite[\S 2.8.4]{Wil09}), which means that $N(T)$ is the disjoint union of at most $r!2^{r}$ varieties of the form $n_{i}T$. Any two such varieties that actually have $\mathbb{F}_{q}$-points must have the same number of $\mathbb{F}_{q}$-points: in fact, for any two $\mathbb{F}_{q}$-points $x,y$ lying on each one of them, the multiplication maps by $x^{-1}y$ and by $y^{-1}x$ are morphisms defined over $\mathbb{F}_{q}$. Therefore, each $(n_{i}T)(\mathbb{F}_{q})$ either is empty or has size $|T(\mathbb{F}_{q})|$, and the bound for $\mathcal{W}(G)$ applies to $|N(T)(\mathbb{F}_{q})|/|T(\mathbb{F}_{q})|$ as well. 
\end{proof}

We must ensure that there are not too many elements that are not {\em regular semisimple}. Regular semisimple elements are dense in $G$ over $\overline{\mathbb{F}_{q}}$, but we need results over $\mathbb{F}_{q}$. We start by showing that there is at least one such element, and then use Proposition~\ref{pr:escape}.

\begin{proposition}\label{pr:existrs}
Let $G$ be an untwisted classical group over $\mathbb{F}_{q}$. Then there exists a regular semisimple element in $G(\mathbb{F}_{q})$.
\end{proposition}

\begin{proof}
By \cite[Prop.~7.1.4, Rem.~7.1.5]{DOR10}, this fact is true even for every $G$ reductive, with some exceptions in the case $q=2$ that can be verified independently.\footnote{We thank user Jeff Adler on MathOverflow for this reference, and for pointing out that the case $G_{2}(\mathbb{F}_{2})$ was also inadvertently omitted (question 452013).} A more indirect approach that covers all cases is the following: by \cite[Cor.~3.5]{Leh92} every simply connected semisimple $G$ has an odd (and thus positive) number of rational regular semisimple conjugacy classes, i.e.\ classes made of regular semisimple elements and fixed (setwise) by the Frobenius automorphism, and by \cite[Thm.~1.7]{Ste65} each such class has an element over $\mathbb{F}_{q}$.
\end{proof}

\begin{proposition}\label{pr:nonrs}
Let $G$ be an untwisted classical group of rank $r$ over $\mathbb{F}_{q}$, and let
\begin{equation*}
\mathfrak{B}=\{g\in G(\mathbb{F}_{q}):g\text{ not regular semisimple}\}.
\end{equation*}
Then $\mathfrak{B}=(G\cap W)(\mathbb{F}_{q})$, where $W$ is a variety of degree $\leq 6r^{2}$ with $G\cap W\subsetneq G$. For any set $A\subseteq G(\mathbb{F}_{q})$ with $\langle A\rangle=G(\mathbb{F}_{q})$, there exists some $g\in A^{k}\setminus(A^{k}\cap\mathfrak{B})$ with $k=(2r)^{O(r^{2})}$.
\end{proposition}

\begin{proof}
An element $g\in G(\mathbb{F}_{q})\subseteq\mathrm{SL}_{n}(\mathbb{F}_{q})$ is regular semisimple if and only if the discriminant of its characteristic polynomial is nonzero; recall that in the case $G=\mathrm{SL}_{n}$ we use the conventional way to represent $G$ as embedded in $\mathrm{Mat}_{n}$, so that we have $n\leq 2r+1$ for all $G$. Call this discriminant $P$. Then, $P$ is of degree $n(n-1)$ in the entries of $g$, so the variety $W=\{P=0\}$ is such that $\mathfrak{B}=(G\cap W)(\mathbb{F}_{q})$ and $\deg(W)\leq\deg(P)=n(n-1)\leq 2r(2r+1)\leq 6r^{2}$.

By Proposition~\ref{pr:existrs}, every $G$ contains a regular semisimple element over $\mathbb{F}_q$, so $\mathfrak{B}\subsetneq G(\mathbb{F}_{q})$ and thus $G\cap W\subsetneq G$. Therefore we can use Proposition~\ref{pr:escape} on $W$, and the number of steps in which we find a regular semisimple $g$ is bounded by $n^{2}\deg(W)^{n^{2}}=(2r)^{O(r^{2})}$.
\end{proof}

For $G$ defined over $K$ and $g\in G(\overline{K})$, we define the {\em conjugacy class} $\mathrm{Cl}(g)$ to be the image of $G$ through the map $x\mapsto xgx^{-1}$.

\begin{proposition}\label{pr:vardeg}
Let $G\leq\mathrm{GL}_{n}$ be an untwisted classical group of rank $r$ with $\dim(G)=\delta$ and $\deg(G)=\Delta$, defined over $\mathbb{F}_{q}$.
\begin{enumerate}[(i)]
\item\label{pr:vardeg-cl} For any regular semisimple $g\in G(\mathbb{F}_{q})$, $\mathrm{Cl}(g)$ is a variety over $\mathbb{F}_{q}$ with $\dim(\mathrm{Cl}(g))=\delta-r$ and $\deg(\mathrm{Cl}(g))\leq n!\Delta\leq 2^{14r^{2}}$. 
\item\label{pr:vardeg-to} For any maximal torus $T$, $\deg(T)\leq 2^{r}$.
\end{enumerate}
\end{proposition}

The constants in the statement above are far from tight. We only need qualitatively good bounds for Theorem~\ref{th:growth}.

\begin{proof}
Since $g$ is semisimple, $\mathrm{Cl}(g)$ is a variety \cite[\S 1.7]{Hum95a}, and since $g$ is regular the dimension of the centralizer $C(g)$ is $r$ \cite[\S 2.3]{Hum95a}, which implies $\dim(\mathrm{Cl}(g))=\delta-\dim(C(g))=\delta-r$ \cite[\S 1.5]{Hum95a}. Using \cite{Fre52} and dividing by an appropriate root of the determinant, we see that two elements $g,g'\in G(\mathbb{F}_{q})$ are conjugates via an element of $G(\overline{\mathbb{F}_{q}})$ if and only if they are conjugates via an element of $\mathrm{GL}_{n}(\overline{\mathbb{F}_{q}})$. In turn, $g,g'$ are conjugates in $\mathrm{GL}_{n}(\overline{\mathbb{F}_{q}})$ if and only if they have the same characteristic polynomial $p(t)$. The coefficient of $t^{i}$ in $p(t)$ has degree $n-i$ in the entries of $g$, so we obtain $\deg(\mathrm{Cl}(g))\leq n!\Delta$.
Finally, by the definitions in \S\ref{se:chev} (except that we adopt the usual embedding $\mathrm{SL}_{n}\leq\mathrm{GL}_{n}$) and by B\'ezout, we have $n\leq 2r+1$ and $\Delta\leq 2^{n^{2}}n$, proving \eqref{pr:vardeg-cl}.

All maximal tori are conjugate to each other over $\overline{\mathbb{F}_{q}}$ \cite[Cor.~11.3(1)]{Bor91}, so they all have the same degree; to prove \eqref{pr:vardeg-to}, it is then sufficient to pick a convenient maximal torus for each $G$. 
The diagonal torus in $\mathrm{SL}_{r+1}\leq\mathrm{GL}_{r+1}$ has degree $\leq r+1$.
The diagonal tori in $\mathrm{Sp}_{2r}$ and $\mathrm{SO}_{2r}^{+}$ consist of diagonal matrices with diagonal entries $x_1,x_2,\dotsc,x_r,x_1^{-1},\dotsc,x_r^{-1}$; they thus have degree $\leq 2^{r}$. The diagonal torus in $\mathrm{SO}_{2r+1}$ is obtained by adding to the previous one a final $1$ along the diagonal, and therefore has degree $\leq 2^{r}$.
\end{proof}

\subsection{Growth and diameter bounds}

Now we can use the degree bounds of Propositions~\ref{pr:nonrs}--\ref{pr:vardeg} and apply Theorem~\ref{th:main} (or rather, its more effective versions from \S\ref{se:dimest}) to the varieties that interest us. Let $g$ be regular semisimple, let $T$ be a maximal torus containing a regular semisimple element, and let $W$ be as in Proposition~\ref{pr:nonrs}, so that in particular $T\cap W\subsetneq T$ and $(T\cap W)(\mathbb{F}_{q})=T(\mathbb{F}_{q})\cap\mathfrak{B}=\{x\in T(\mathbb{F}_{q}):x\text{ not regular semisimple}\}$. We apply Theorem~\ref{th:ober} to the varieties $V=\mathrm{Cl}(g)$ and $V=T\cap W$ for $G\neq\mathrm{SL}_{n}$, and Remark~\ref{re:slnusual} for $G=\mathrm{SL}_{n}$. Hence, for any untwisted classical group $G$ of rank $r$ over $\mathbb{F}_{q}$, any generating set $A$ of $G(\mathbb{F}_q)$ with $e\in A$, and any $m\geq 1$,
\begin{align}\label{eq:de}
|A^m\cap\mathrm{Cl}(g)(\mathbb{F}_{q})| & \leq C_{1}|A^{C_{2}(m)}|^{1-\frac{r}{\dim(G)}}, & |A^m\cap T(\mathbb{F}_{q})\cap\mathfrak{B}| & \leq C_{1}'|A^{C_{2}(m)}|^{\frac{r-1}{\dim(G)}},
\end{align}
where, by Propositions \ref{pr:nonrs} and \ref{pr:vardeg}, we may choose
\begin{align*}
C_{1} & = (2^{n^{2}+1}\dim(G)\cdot\deg(\Cl(g)))^{\dim (G)^{\dim(\Cl(g))}} = 
(2^{O(r^2)})^{O(r^2)^{\dim (G) - r}} = 2^{2^{O(r^2 \log r)}}, \\
C_1' & = (2^{n^{2}+1}\dim(G)\cdot\deg(T\cap W))^{\dim (G)^{\dim(T\cap W)}} 
= (2^{O(r^{2})})^{O(r^2)^r} = 2^{2^{O(r \log r)}}, \\
C_{2}(m) & \leq \dim(G) (m + 2^{n^2+1}) = 2^{O(r^{2})} + O(r^2 m).
\end{align*}
As anticipated, we will obtain a product theorem of the form \eqref{eq:prodthm} from \eqref{eq:de}.

First -- a result by Nikolov and Pyber \cite{NP11} makes short work of the case of $A$ large.

\begin{proposition}\label{pr:np11}
Let $G$ be an untwisted classical group of rank $r$ with $\dim(G)=\delta$, defined over $\mathbb{F}_{q}$, and let $A\subseteq G(\mathbb{F}_{q})$. If $|A|\geq 3q^{\delta-\frac{r}{3}}$, then $A^{3}=G(\mathbb{F}_{q})$.
\end{proposition}

\begin{proof}
Let $k$ be the degree of a minimal complex representation of $G(\mathbb{F}_{q})$. If $|A|>k^{-\frac{1}{3}}|G(\mathbb{F}_{q})|$ then $A^{3}=G(\mathbb{F}_{q})$ by \cite[Cor.~1]{NP11}. We have $|G(\mathbb{F}_{q})|\leq 2q^{\delta}$ as mentioned in \S\ref{se:chev}, whereas $k\geq\frac{1}{3}q^{r}$ by \cite[Table II]{TZ96}. The result follows.
\end{proof}

And now we prove our product theorem. Recall that, by \S\ref{se:chev}, the expression ``untwisted classical group'' excludes cases with particularly small rank: this means that pathological small cases are avoided, such as $\SO_{2}^{+}$ which, being abelian, cannot satisfy a product theorem as the one below.

\begin{theorem}\label{th:growth}
Let $G$ be an untwisted classical group of rank $r$ with $\dim(G)=\delta$, defined over $\mathbb{F}_{q}$.
Let $A\subseteq G(\mathbb{F}_{q})$ with $e\in A=A^{-1}$ and $\langle A\rangle=G(\mathbb{F}_{q})$. Let $m\geq 1$.
Then, either $A^{6 m}=G(\mathbb{F}_{q})$ or $|A^{\ell_1 m + \ell_0}|\geq\frac{1}{c}|A^{m}|^{1+\eta}$ with
\begin{align*}
c & =2^{2^{O(r^{2}\log r)}}, & \ell_1 & = O(r^2), &\ell_0 &=(2 r)^{O(r^2)}, & \eta & =\frac{1}{O(r^{2})}.
\end{align*}
\end{theorem}
The proof follows the standard pivoting argument we mentioned in
\S \ref{subs:appl}.
\begin{proof}
Let $k$ be as in Proposition~\ref{pr:nonrs}; in particular,
$k=(2r)^{O(r^{2})}$, and we can also assume $k\geq 2$. Call a maximal torus $T$ of $G$ \textit{involved} if $A^{k}\cap T(\mathbb{F}_{q})$ contains a regular semisimple element. By Proposition~\ref{pr:nonrs}, there exists an involved torus. Two cases arise.
\begin{enumerate}[{Case} 1:]
\item\label{it:invnoinv} There are a maximal torus $T$ and an element $h\in A$ such that $T$ is involved and $hTh^{-1}$ is not involved.
\item\label{it:allinv} There is a maximal torus $T$ such that all maximal tori $hTh^{-1}$ are involved for all $h\in G(\mathbb{F}_{q})$.
\end{enumerate}

(Note how this is an example of how one may apply an inductive argument without a natural ordering. All we are using is the fact that $A$ generates $G(\mathbb{F}_q)$.)

\subsubsection*{Case~\ref{it:invnoinv}.} Let $\mathfrak{B}$ denote the set of elements that are not regular semisimple, as in Proposition~\ref{pr:nonrs}. Since $T$ is involved, there is some $g\in A^{k}\cap T(\mathbb{F}_{q})$ with $g\notin\mathfrak{B}$. Let $g'=hgh^{-1}\in A^{k+2}\cap T'(\mathbb{F}_{q})$, which defines a map
\begin{align}\label{eq:psidef}
\psi_{T'} & :A^m\rightarrow\Cl(g')(\mathbb{F}_{q}), & a & \mapsto ag'a^{-1}.
\end{align}
If $\psi_{T'}(a_{1})=\psi_{T'}(a_{2})$ then $a_{1}^{-1}a_{2}$ commutes with $g'$; since $g'$ is regular,  $C(g')=T'$, and so $a_{1}^{-1}a_{2}\in A^{2 m}\cap T'(\mathbb{F}_{q})$ (this can be done because we now require $A=A^{-1}$). Hence $a_{1}^{-1}a_{2}\in\mathfrak{B}$, since $T'$ is not involved. Thus, the size of every fibre of $\psi_{T'}$ is bounded by $|A^{2 m}\cap T'(\mathbb{F}_{q})\cap\mathfrak{B}|$. Combining this bound with~\eqref{eq:de}, we get
\begin{equation*}
\frac{|A^m|}{C_{1}' |A^{2C_{2}(m)}|^{\frac{r-1}{\delta}}}\leq\frac{|A^m|}{|A^{2 m}\cap T'(\mathbb{F}_{q})\cap\mathfrak{B}|}\leq|A^{k+2+2 m}\cap\Cl(g')(\mathbb{F}_{q})|\leq C_{1}|A^{C_{2}(k+2 m + 2)}|^{1-\frac{r}{\delta}},
\end{equation*}
which, since we may assume $C_2(m)$ is increasing on $m$, yields
\begin{equation*}
|A^{C_2(k+2 m + 2)}|^{1-\frac{1}{\delta}}\geq\frac{1}{C_{1} C_1'}|A|.
\end{equation*}
As mentioned in \S\ref{se:chev} we have $\delta=O(r^{2})$, giving $|A^{m'}|\geq\frac{1}{c}|A^m|^{1+\eta}$ for
\begin{align*}
c & = (C_{1} C_1')^{\frac{1}{1-1/\delta}}=2^{2^{O(r^{2}\log r)}}, & \eta & =\frac{1}{\delta-1}=\frac{1}{O(r^{2})}.
\end{align*}
\begin{align*}
 m' = C_2(k+2 m + 2) = 
 2^{O(r^2)} + O(r^2 (k + 2 m + 2)) = (2 r)^{O(r^2)} + O(r^2 m).
\end{align*}

(In other words: since $h T h^{-1}$ is not involved, $h g h^{-1}$ has very
many distinct conjugates $a h g h^{-1} a^{-1}$, $a\in A^m$, as any element of $(A^m)^{-1} A^m$ commuting with $h g h^{-1}$ would have to be of a special shape. At the same time,
$a h g h^{-1} a^{-1}\in A^{k+2m+2}\cap \Cl(h g h^{-1})$, and we know from our dimensional
estimate that $A^{k+2 m + 2}\cap \Cl(h g h^{-1})$ cannot have too many elements -- unless
$A^{C_2(k+2 m+2)}$ is much larger than $A$. Hence,
$A^{C_2(k+ 2 m+ 2)}$ is much larger than $A$.)

\subsubsection*{Case~\ref{it:allinv}.} For each $T'$ of the form $hTh^{-1}$, arbitrarily fix an element $g'\in A^{k}\cap T'(\mathbb{F}_{q})$ with $g'\notin\mathfrak{B}$ (i.e.\ regular semisimple), and define $\psi_{T'}$ as in \eqref{eq:psidef}. 
Then $\psi_{T'}(A^m)\subseteq A^{k+2 m}\cap\mathrm{Cl}(g')(\mathbb{F}_{q})$ for each $T'$.
As in Case~\ref{it:invnoinv}, the bound $|\psi_{T'}^{-1}(x)|\leq|A^{2 m}\cap T'(\mathbb{F}_{q})|$ holds for every fibre. Thus,
\begin{equation}\label{eq:fibre}
|A^{2 m}\cap T'(\mathbb{F}_{q})|\geq\max_{x\in\psi_{T'}(A
^m)}|\psi_{T'}^{-1}(x)|\geq\frac{|A^m|}{|\psi_{T'}(A^m)|}\geq\frac{|A^m|}{|A^{k+2 m}\cap\mathrm{Cl}(g')(\mathbb{F}_{q})|}.
\end{equation}
Now, every regular semisimple element of $G$ sits inside a unique maximal torus, so that $G(\mathbb{F}_{q})\setminus(G(\mathbb{F}_{q})\cap\mathfrak{B})$ is partitioned between the maximal tori of $G$. On the other hand, the size of the set $\mathcal{T}$ of conjugates of $T$ can be bounded via Lemma~\ref{le:hth}. Therefore, by \eqref{eq:fibre} and \eqref{eq:de},
\begin{align}
|A^{2 m}| & \geq|A^{2 m}\setminus(A^{2 m}\cap\mathfrak{B})|\geq\sum_{T'\in\mathcal{T}}(|A^{2 m}\cap T'(\mathbb{F}_{q})|-|A^{2 m}\cap T'(\mathbb{F}_{q})\cap\mathfrak{B}|) \nonumber \\
 & \geq\frac{q^{\delta-r}}{(6 r)^{r}}\left(\frac{|A^m|}{C_{1}|A^{C_{2}(k+2 m)}|^{1-\frac{r}{\delta}}}-C_{1}|A^{C_{2}(2 m)}|^{\frac{r-1}{\delta}}\right). \label{eq:oofdiff}
\end{align}
(We can use \eqref{eq:de} because each $T'$ is involved, and thus it has a regular semisimple element.) If we suppose that the second term in \eqref{eq:oofdiff} is large, i.e.,
\begin{equation}\label{eq:oofhyp}
C_{1}|A^{C_{2}(2m)}|^{\frac{r-1}{\delta}}\geq\frac{|A^m|}{2C_{1}|A^{C_{2}(k+ 2 m)}|^{1-\frac{r}{\delta}}},
\end{equation}
then
\begin{equation}\label{eq:oofthen}
|A^{C_{2}(k + 2 m)}|\geq\left(\frac{|A^m|}{2C_{1}^{2}}\right)^{\frac{1}{1-\frac{1}{\delta}}}\geq\frac{|A^m|^{1+\frac{1}{\delta}}}{(2C_{1}^{2})^{1+\frac{2}{\delta}}}
\end{equation}
since $\delta \geq 2$. Suppose instead that \eqref{eq:oofhyp} does not hold. Then,
by \eqref{eq:oofdiff},
\[|A^{2 m}|\geq 
\frac{q^{\delta-r}}{(6 r)^{r}}
\frac{|A^m|}{2 C_{1}|A^{C_{2}(k + 2 m)}|^{1-\frac{r}{\delta}}}.
\]
If $|A^{2 m}|\geq 3 q^{\delta-\frac{r}{3}}$, then, by Proposition~\ref{pr:np11}, $A^{6 m}=G(\mathbb{F}_{q})$. Assume otherwise. Then
\begin{equation}\label{eq:tftt}
|A^{C_2(k+2 m)}|^{1 - \frac{r}{\delta}} |A^{2 m}|^{1 - \frac{\delta-r}{\delta-r/3}}
\geq \left(\frac{q^{\delta- \frac{r}{3}}}{|A^{2 m}|}\right)^{\frac{\delta-r}{\delta-r/3}} \frac{|A^m|}{2 C_1 (6 r)^r}
\geq \frac{|A^m|}{6 C_1 (6 r)^r}.\end{equation}
By the hypotheses on the rank and the values of $\delta$ in \S\ref{se:chev} we have $r/\delta \leq 1/3$, and so 
\[1-\frac{r}{\delta} + 1 - \frac{\delta-r}{\delta-r/3} = 
1 - \frac{r - r^2/3\delta - 2 r/3}{\delta-r/3} \leq 1 - \frac{2r/9}{\delta}.\]
We may assume $C_2(k+2m)\geq 2m$. Hence, by \eqref{eq:tftt},
\begin{equation}\label{eq:oofelse}|A^{C_2(k+2 m)}| \geq
\frac{|A^m|^{1+\frac{2 r}{9 \delta}}}{(6 C_1 (6 r)^r)^{(1-2r/9\delta)^{-1}}}.
\end{equation}
Again by $r/\delta\leq 1/3$, $1-2 r / 9\delta \geq 25/27$.

Since $\delta=O(r^{2})$ by \S\ref{se:chev}, both \eqref{eq:oofthen} and \eqref{eq:oofelse} imply that $|A^{m'}|\geq\frac{1}{c}|A|^{1+\eta}$ with
\begin{align*}
c & =\max\left\{(2C_{1}^{2})^{1+\frac{2}{\delta}},(6 C_1 (6 r)^r)^{\frac{27}{25}}
\right\}=2^{2^{O(r^{2}\log r)}}, 
& \eta & =\min\left\{\frac{1}{\delta},\frac{2 r}{9 \delta}\right\}=\frac{1}{O(r^{2})}.
\end{align*}
\begin{align*}
m' & = C_{2}(k+2 m) =2^{O(r^{2})} + O(r^2 (k + 2 m)) = (2 r)^{O(r^2)} + O(r^2 m). \qedhere
\end{align*}
\end{proof}

We now move to the proof of a diameter bound of the form \eqref{eq:diam}.

\begin{theorem}\label{th:diam}
Let $G$ be an untwisted classical group of rank $r$ defined over $\mathbb{F}_{q}$. Let $A\subseteq G(\mathbb{F}_{q})$ with $\langle A\rangle=G(\mathbb{F}_{q})$. Then
\begin{equation*}
\mathrm{diam}(G(\mathbb{F}_{q}),A)=2^{2^{O(r^{2}\log r)}}(\log|G(\mathbb{F}_{q})|)^{O(r^{2}\log r)}.
\end{equation*}
\end{theorem}

\begin{proof}
We can assume without loss of generality that $e\in A$, by definition of diameter. We may also assume that $A=A^{-1}$, thanks to \cite[Thm.~1.4]{Bab06}.

By Theorem~\ref{th:growth}, for any $m\geq 1$, we have either $A^{6 m}=G(\mathbb{F}_{q})$ or $|A^{\ell_1 m + \ell_0}|\geq\frac{1}{c}|A^{m}|^{1+\eta}$ with $c,\ell_0,\ell_1,\eta$ as in the statement.

Assume first that $|A|\geq c^{2/\eta}$. Then, for any $m\geq 1$, either $A^{6 m}=G(\mathbb{F}_{q})$ or $|A^{\ell_1 m +\ell_0}|\geq|A^{m}|^{1+\eta/2}$.
Let $f$ be the map $m\mapsto \ell_1 m + \ell_0$, and define
$$m_j = f^j(1) = 
\ell_1^j + (\ell_1^{j-1} + \dotsc + \ell_1 + 1) \ell_0$$
for $j\geq 1$, and $m_0 = 1$.
If $k$ is the smallest non-negative integer for which $A^{6 m_k}=G(\mathbb{F}_{q})$, then we must have $|A^{m_j}|\geq |A|^{(1+\eta/2)^{j}}$ for all $1\leq j\leq k$. Since by hypothesis $|A^{m_{k-1}}|<|G(\mathbb{F}_{q})|$, we get
\begin{equation*}
k<1+\frac{1}{\log(1+\frac{\eta}{2})}\log\left(\frac{\log|G(\mathbb{F}_{q})|}{\log|A|}\right)=O(r^{2}\log\log|G(\mathbb{F}_{q})|),
\end{equation*}
which implies
\begin{equation}\label{eq:diamm}\begin{aligned}
\mathrm{diam}(G(\mathbb{F}_{q}),A) &\leq  6 m_k \leq 6 \ell_1^k (1 + \ell_0) 
\\ &= O(r^2)^{O(r^2 \log \log |G(\mathbb{F}_{q})|)} (2 r)^{O(r^2)}
= (\log|G(\mathbb{F}_{q})|)^{O(r^{2}\log r)}.
\end{aligned}\end{equation}

Now assume that $|A|<c^{2/\eta}$. Then, trivially, $|A^{\lfloor c^{2/\eta}\rfloor}|\geq c^{2/\eta}$. Applying \eqref{eq:diamm} to $A^{\lfloor c^{2/\eta}\rfloor}$ we obtain
\begin{equation*}
\mathrm{diam}(G(\mathbb{F}_{q}),A)\leq  c^{2/\eta} \mathrm{diam}(G(\mathbb{F}_{q}),A^{\lfloor c^{2/\eta}\rfloor})=2^{2^{O(r^{2}\log r)}}(\log|G(\mathbb{F}_{q})|)^{O(r^{2}\log r)},
\end{equation*}
yielding the result.
\end{proof}

\section{Concluding remarks}\label{se:concl}

\subsection{Other finite groups of Lie type}
Since our interest in dimensional estimates stems mainly from questions about finite simple groups, it is natural to ask to what extent our results could apply, with the necessary modifications, to finite groups of Lie type other than the ones given in \S\ref{se:chev}. We observe right away that the {\em orthogonal groups $\mathrm{SO}_{2n}^{+}(\mathbb{F}_q)$ with $q$ even} present no problem. They were excluded in Remark~\ref{re:untwisted}\eqref{re:untwisted-o2} because the algebraic group $G$ is not connected in this case, which presents a few technical inconveniences: for instance, a proper variety $V\subsetneq G$ does not have necessarily $\dim(V)<\dim(G)$. The strategy leading to Theorem~\ref{th:main} is not fundamentally affected though, and tackling this case only adds inessential complications.

Let us turn to the other cases. A {\em finite group of Lie type} consists in the
set of points $G^F$ of an algebraic group $G$ that are fixed by a
{\em Steinberg endomorphism} $F:G\to G$ (the composition of a field automorphism and an automorphism of algebraic groups). The groups we have considered --
groups of the form $G(\mathbb{F}_q)$ -- amount to a special case, namely, the case where the
field automorphism is a Frobenius automorphism $x\mapsto x^q$ and the automorphism of algebraic groups is the trivial one.

The simplest way to generalize our results to a family of
finite groups of Lie type would be to show that there is a linear algebraic
group $G$ of fixed degree for which $G(\mathbb{F}_{q})$ goes over the groups
in that family -- or small extensions thereof.

The {\em untwisted exceptional groups} are covered by Theorem~\ref{th:main}
in this way. Up to extension, each of them is the automorphism group of some finite-dimensional algebra $(V,\circ)$, not necessarily associative, so the elements $g\in G$ are defined by finitely many equations whose coefficients are quadratic in the entries of $g$: concretely, if for a basis $\{e_{i}\}_{i}$ of the algebra we have $g(e_{i})=\sum_{j}g_{ij}e_{j}$ and $e_{i}\circ e_{j}=\sum_{k}\lambda_{ijk}e_{k}$, then the entries of $g$ must satisfy
\begin{equation*}
\sum_{i,j}\lambda_{ijr}g_{si}g_{tj}-\sum_{k}\lambda_{stk}g_{kr}=0
\end{equation*}
for all $r,s,t$. For $G_{2}(\mathbb{F}_{q})$ we can use the octonion algebra, in its classical basis with $2\nmid q$ \cite[\S 4.3.2]{Wil09} and in another convenient one for $2\mid q$ \cite[\S 4.3.4]{Wil09}; for $E_{8}(\mathbb{F}_{q})$, we use its own Lie algebra as there are no smaller representations \cite[\S 4.12.1]{Wil09}. The groups can also be defined as preserving some convenient and explicit forms of bounded degree. We can see $F_{4}$ as the linear algebraic group preserving both a quadratic and a cubic form in $\mathbb{F}_{q}^{26}$ \cite[(4.94)-(4.95)]{Wil09}. Similarly, the linear algebraic group $SE_{6}$ is the stabilizer of a cubic form in $\mathbb{F}_{q}^{27}$, and we have $SE_{6}(\mathbb{F}_{q})\simeq E_{6}(\mathbb{F}_{q})$ for $q\not\equiv 1\ \mathrm{mod}\ 3$ and $SE_{6}(\mathbb{F}_{q})\simeq 3.E_{6}(\mathbb{F}_{q})$ for $q\equiv 1\ \mathrm{mod}\ 3$ \cite[\S\S 4.10.1-4.10.2]{Wil09}. Finally, a quartic form in $\mathbb{F}_{q}^{56}$ defines a linear algebraic group $SE_{7}$ with either $SE_{7}(\mathbb{F}_{q})\simeq E_{7}(\mathbb{F}_{q})$ for $2\mid q$ or $SE_{7}(\mathbb{F}_{q})\simeq 2.E_{7}(\mathbb{F}_{q})$ for $2\nmid q$: see \cite[\S 4.12]{Wil09} and \cite[(2.3)-(4.10)]{Coo95}.

The discussion above leads to the following result, whose constants $O(1)$ can be made completely explicit. They are absolute constants, as these are groups of bounded dimension.

\begin{theorem}\label{th:exceptional}
Let $G$ be an untwisted exceptional group (i.e., $G=SE_{6},SE_{7},E_{8},F_{4},G_{2}$ defined as a linear algebraic group in the way described above), defined over $K=\mathbb{F}_{q}$. Let $A\subseteq G(K)$ be a finite set with $\langle A\rangle=G(K)$ and $e\in A = A^{-1}$. Let $V$ be a subvariety of $G$ defined over $\overline{K}$. Write $d=\dim(V)$, $\delta=\dim(G)$. Then
\begin{equation*}
|A\cap V(\overline{K})|\leq C_{1}|A^{C_{2}}|^{d/\delta},
\end{equation*}
where
\begin{align*}
C_{1} & =O(\deg(V))^{O(1)}, & C_{2} & =O(1).
\end{align*}
\end{theorem}

The {\em twisted groups} are more problematic, because the field automorphism that twists the scalars of the space on which $G$ acts must be included in the definition of $G$. 
If we use the Frobenius automorphism $x\mapsto x^q$ directly, the degree of the map ends up depending on $q$, and consequently so does $\deg(G)$. What we want
is results that are independent of $q$. Hence, we have to find a more indirect
way. We can divide the twisted groups into three types: {\em twisted classical groups} ($\mathrm{SU}_{n}$ and $\mathrm{SO}_{2n}^{-}$, obtainable by twisting a group among those in \S\ref{se:chev}), {\em Steinberg-Tits-Hertzig groups} (as they are called in \cite[\S 4.1]{Wil09}, comprising $^{2}E_{6}$ and $^{3}D_{4}$), and {\em Suzuki-Ree groups} ($^{2}B_{2}$, $^{2}G_{2}$, and $^{2}F_{4}$). We manage to say something for the first two types.

Consider first the unitary group: $\mathrm{SU}_{n}(q)$ is definable as the set of matrices $g\in\mathrm{SL}_{n}(\mathbb{F}_{q^{2}})$ satisfying $g^{-1}=(g^{q})^{\top}$, where $g^{q}$ is obtained by raising every entry of $g$ to the $q$-th power.
Now, just like we could deal with $\SU_n(\mathbb{C})$ by considering it as a real group (decomposing $\mathbb{C} = \mathbb{R} + i \mathbb{R}$), we can deal
with $\SU_n(q)$ by expressing it as the set of rational points of a group
over $\mathbb{F}_q$. Let us do this in detail.

Let $q$ be odd for simplicity, and fix a quadratic nonresidue $s\in\mathbb{F}_{q}$: thus $\mathbb{F}_{q^{2}}\simeq\mathbb{F}_{q}[x]/(x^{2}-s)$, and if $\zeta$ is a square root of $s$ inside $\mathbb{F}_{q^{2}}$ we have $\zeta^{q}=s^{(q-1)/2}\zeta=-\zeta$. The idea should now be clear: we take the injective function $\varphi:\mathrm{GL}_{n}(\mathbb{F}_{q^{2}})\rightarrow\mathrm{GL}_{2n}(\mathbb{F}_{q})$, where the matrix $\varphi(g)$ is given by replacing each entry $a+b\zeta\in\mathbb{F}_{q^{2}}$ of $g$, where $a,b\in\mathbb{F}_{q}$, with the $2\times 2$ submatrix $\begin{pmatrix}a&sb\\b&a\end{pmatrix}$. The function $\varphi$ is not polynomial but it preserves matrix multiplication, so it induces a group isomorphism from $\mathrm{SU}_{n}(q)$ to its image. Furthermore, note that $(a+b\zeta)^{q}=a+b\zeta^{q}=a-b\zeta$: therefore, there is a new twist $\tau:\mathrm{GL}_{2n}\rightarrow\mathrm{GL}_{2n}$ defined by linear polynomials $\tau(x)_{ij}=(-1)^{i+j}x_{ij}$ and such that $\tau(\varphi(g))=\varphi(g^{q})$. Finally, one verifies that $\tau(x)\tau(y)=\tau(xy)$ and $\tau(x)^{-1}=\tau(x^{-1})$ whenever $x$ is invertible. Hence, there is a linear algebraic group $H\leq\mathrm{GL}_{2n}$ with $H(\mathbb{F}_{q})\simeq\mathrm{SU}_{n}(q)$. The equations defining $H$ either come from $x^{-1}=\tau(x)^{\top}$, or represent the condition $\det(g)=1$ for $x=\varphi(g)$, or are of the form $x_{2i-1,2i-1}=x_{2i,2i}$ or $x_{2i-1,2i}=sx_{2i,2i-1}$: this means that $\deg(H)$ does not depend on $q$, although $H$ itself does (via $s$). Using this representation of $\mathrm{SU}_{n}$, we are able to conclude that Theorem~\ref{th:main} applies to $\SU_n$ when $q$ is odd. When $q$ is even, we just have to choose a different irreducible polynomial in lieu of $x^{2}-s$ and repeat the procedure with the due modifications.

Just as we twist $\mathrm{SL}_{n}$ to obtain $\mathrm{SU}_{n}$, we may twist $\mathrm{SO}_{2n}^{+}$ and $\mathrm{SO}_{2n+1}$ to obtain $\mathrm{SO}_{2n}^{-}$ \cite[p.~xiv]{CCNPW85}. The procedure works in the exact same way, and we obtain $\mathrm{SO}_{2n}^{-}(\mathbb{F}_{q})$ as the set of $\mathbb{F}_{q}$-points of a linear algebraic group of degree independent from $q$. Often this group is not even defined using such a twist, and it comes instead directly from the orthogonal group of a form of minus type: see for example \cite[\S 2.5]{KL90}, \cite[\S 3.7]{Wil09}, and Remark~\ref{re:untwisted}\eqref{re:untwisted-minus}.

We may also twist $E_{6}$ to obtain $^{2}E_{6}$ \cite[\S 4.11]{Wil09}. In this case we obtain a linear algebraic group $^{2}SE_{6}$ of degree independent from $q$, and the set of its $\mathbb{F}_{q}$-points is either $^{2}E_{6}(q)$ or $3.^{2}E_{6}(q)$ (for $q\not\equiv 1\ \mathrm{mod}\ 3$ and $q\equiv 1\ \mathrm{mod}\ 3$ respectively). A similar method works for $^{3}D_{4}$, although this time we need the field $\mathbb{F}_{q^{3}}$: working in the triple product of the octonion algebra over $\mathbb{F}_{q^{3}}$, $^{3}D_{4}(q)$ can be seen as the subgroup of isotopies commuting with the map $(x,y,z)\mapsto(y^{q},z^{q},x^{q})$ (composing the triality automorphism and the field automorphism, see \cite[\S 4.7.2]{Wil09}). Fixing as before an element $\zeta\in\mathbb{F}_{q^{3}}\setminus\mathbb{F}_{q}$ with $\zeta^{3}=s_{0}+s_{1}\zeta+s_{2}\zeta^{2}$ and $s_{i}\in\mathbb{F}_{q}$, we can pass from $\mathrm{GL}_{n}(\mathbb{F}_{q^{3}})$ to $\mathrm{GL}_{3n}(\mathbb{F}_{q})$ via the entry-wise replacement
\begin{equation*}
a+b\zeta+c\zeta^{2}\mapsto\begin{pmatrix} a & s_{0}c & s_{0}b+s_{0}s_{2}c \\ b & a+s_{1}c & s_{1}b+(s_{0}+s_{1}s_{2})c \\ c & b+s_{2}c & a+s_{2}b+(s_{1}+s_{2}^{2})c \end{pmatrix},
\end{equation*}
and the rest goes much as before.

Thus, all twisted classical and Steinberg-Tits-Hertzig groups can be covered by the methods of Theorem~\ref{th:main}. A word of warning is in order here: in the statement we
are about to give, $d$, $\delta$ and $\deg (V)$ refer to dimensions and degrees in the embeddings given above. Thus, $d$ and $\delta$ are twice what the dimensions would
be in $\SU_n$, $\SO_{2 n}^-$ or $^{2}\!SE_{6}$ in their usual definitions, and thrice
what they would be in $^{3}\!D_{4}$ in its usual definition.

\begin{theorem}\label{th:twisted}
Let $G$ be a twisted non-Suzuki-Ree group (i.e., $G=\mathrm{SU}_{n},\mathrm{SO}_{2n}^{-},^{2}\!SE_{6},^{3}\!D_{4}$ defined as a linear algebraic group in the way described above), of rank $r$ defined over $K=\mathbb{F}_{q}$. Let $A\subseteq G(K)$ be a finite set with $\langle A\rangle=G(K)$ and $e\in A = A^{-1}$. Let $V$ be a subvariety of $G$ defined over $\overline{K}$. Write $d=\dim(V)$, $\delta=\dim(G)$. Then
\begin{equation*}
|A\cap V(\overline{K})|\leq C_{1}|A^{C_{2}}|^{d/\delta},
\end{equation*}
where
\begin{align*}
C_{1} & =(2\delta \deg(V))^{\delta^d}, & C_{2} & =e^{O(r^{2} \log r)}.
\end{align*}
\end{theorem}

The remaining groups (namely, Suzuki-Ree groups) do not come from automorphisms of order $2$, as in $\mathrm{SU}_{n}$, $\mathrm{SO}_{2n}^{-}$, and $^{2}E_{6}$, or $3$, as in $^{3}D_{4}$. The twist comes from the square root of the Frobenius automorphism on the field with $2^{2n+1}$ elements for $^{2}B_{2}$ and $^{2}F_{4}$, and with $3^{2n+1}$ elements for $^{2}G_{2}$: see respectively \cite[\S\S 4.2.1-4.9.1-4.5.1]{Wil09}. However, this means that the order of the automorphisms is $2n+1$, and the procedure we used before leads to entry-wise replacements of size $(2n+1)\times(2n+1)$: thus, the degree of the algebraic group depends on the field.

A less na\"ive approach may work. What one would need to do is
generalize \S \ref{subs:pointsoutV} (in particular, Propositions~\ref{pr:lwnotq}--\ref{pr:lwlowg}) to count points that are fixed by a given endomorphism and lie
on a given variety.
More precisely, we would need to show that, for $G$ an almost simple linear algebraic group with a Steinberg endomorphism $F:G\to G$, and $V$ a subvariety of $G$, then there is at least one point on $G^F$ {\em not} on $V$, provided that
$G^F$ has more than $C$ elements, where $C$ depends only on the degree of $V$.
We would of course like the dependence of $C$ on $\deg (V)$ to be explicit, and
for $C$ not to grow too rapidly as $\deg (V)$ grows.

\subsection{Jordan's theorem in positive characteristic}

As we said in the introduction, Larsen and Pink proved their dimensional estimates \cite{LP11} 
with one main application in mind, namely, generalizing Jordan's theorem to positive characteristic $p$ without going through the classification of finite simple groups.
More precisely: \cite[Thm.~0.3]{LP11} states that, given a simple subgroup $\Gamma$ of
$\GL_n(\mathbb{F}_q)$, either $\Gamma$ is bounded by a constant depending only on $n$,
or $\Gamma$ is a group of Lie type. It is natural to ask what the applications of
our work to results of that kind could be.

Giving an optimized explicit version of \cite[Thm.~0.3]{LP11} would overstep the bounds
of the present paper. Let us nevertheless give a very quick outline of a possible strategy towards such a result.
We have not yet carried it out in full, but we are considering doing so in a possible future article.

In Theorem \ref{th:growth} and elsewhere, the assumption $\langle A\rangle = G(\mathbb{F}_q)$ can be weakened. That assumption is mainly used in escape from
subvarieties, and we can see that the escape result that we use (Prop.~\ref{pr:escape}) does not really require it in full strength. Rather, what we need to assume, for several
varieties $V$ arising in the course of the proof, is that $\langle A\rangle x
\not\subseteq V(K)$, i.e., $A$ is not contained in the stabilizer group of $V(K)$.

Now let $A = \Gamma$ and apply Theorem \ref{th:growth} with that weakened assumption
and $m=1$. Clearly, $A^{6 m} = A^6 = A$ and $A^{\ell_1 m + \ell_0} = A^m = A$. Hence,
either $A = G(\mathbb{F}_q)$ or $|A|$ is bounded by the constant $c^{1/\eta}$. There is a third possibility,
namely, $A$ may not fulfill the weakened assumption. In that case, $A$ is contained
in the stabilizer group of $V(K)$, which is a proper linear algebraic subgroup $H$ of $G$. 
We should then be able to recur, using $H$ instead of $G$.

Of course we should still be careful, as while $H$ will be a linear algebraic group,
it may not be an untwisted classical group, as in Theorem \ref{th:growth}, or
even just connected and almost simple (which is all we really need),
as in Theorem \ref{th:main}. We should nevertheless be able to reduce matters to
the connected, almost simple case, using the assumption that $\Gamma$ is simple.


\section*{Acknowledgements}

JB was supported by ERC Consolidator grants 648329 (codename GRANT, with H.~Helfgott as PI) and 681207 (codename GrDyAP, with A.~Thom as PI). DD has been supported by ERC Consolidator grant 648329, the Israel Science Foundation Grants No. 686/17 and 700/21 of A.~Shalev, the Emily Erskine Endowment Fund, and ERC Advanced grant 741420 (codename GROGandGIN, with L.~Pyber as PI). HH was supported by ERC Consolidator grant 648329 and by his Humboldt professorship.

The authors would like to thank their collaborators in G\"ottingen and the community of MathOverflow for suggestions and insight. Thanks are also due to
Miles Reid and Dominic Bunnett for discussions on Bézout's theorem and
different meanings of the word ``intersection'' and to Inna Capdeboscq for helping us understand the differences between how some groups are understood in the study of linear algebraic groups and in finite group theory. Thanks are also due to two anonymous referees,
in part for drawing our attention to those differences.

\nocite{}
\bibliographystyle{abbrv}
\bibliography{BDH}
\end{document}